%% file: main.tex
\documentclass[11pt,a4paper]{article}

\input{packages.tex}
\input{macros.tex}
\title{Conditional validity and a fast approximation formula of full conformal prediction sets}
\author[1]{Nicolai Amann}
\affil[1]{University of Vienna}

\begin{document}
    \selectlanguage{english}
    \maketitle
    \input{abstract.tex}
	\input{paper.tex}
    \printbibliography[heading=bibintoc,title=References]
    \appendix{}
    \input{appendix.tex}
\end{document}

%% file: packages.tex
\usepackage[utf8]{inputenc}
\usepackage[ngerman, english]{babel}

\usepackage{amsmath}
\usepackage{amsfonts}
\usepackage{amssymb}
\usepackage{dsfont}
\usepackage{amsthm}
\usepackage{fancyhdr}
\usepackage[textheight=600pt]{geometry}
\usepackage[hang,flushmargin]{footmisc}
\usepackage[style=ext-authoryear, citestyle=authoryear, articlein=false, backend=biber, uniquename=init, giveninits=true, maxbibnames=10, maxcitenames=2]{biblatex}
\usepackage{enumitem}

\usepackage{csquotes}
\usepackage{comment}
\usepackage{setspace}
\usepackage{graphicx}
\usepackage{ifthen}

\usepackage{hyperref}
\hypersetup{breaklinks=true}

\usepackage[noabbrev]{cleveref}
\usepackage{authblk}
\usepackage[dvipsnames]{xcolor}
\usepackage[linecolor=OliveGreen,backgroundcolor=OliveGreen!25,bordercolor=OliveGreen,textsize=tiny]{todonotes}

\usepackage{algorithm}
\usepackage{algpseudocode}

%% file: macros.tex
\theoremstyle{plain}
\newtheorem{theorem}{Theorem}[section]
\newtheorem*{theorem*}{Theorem}
\newtheorem{lemma}[theorem]{Lemma}
\newtheorem{proposition}[theorem]{Proposition}
\newtheorem{corollary}[theorem]{Corollary}

\theoremstyle{definition}
\newtheorem{remark}[theorem]{Remark}
\newtheorem{definition}[theorem]{Definition}
\newtheorem{example}[theorem]{Example}

\newcommand{\E}{{\, \mathbb E}}
\newcommand{\R}{\mathbb{R}}
\newcommand{\N}{\mathbb{N}}
\newcommand{\Z}{\mathbb{Z}}
\renewcommand{\P}{\mathbb{P}}
\newcommand{\PC}[2]{\mathbb{P} \left( #1 \| #2 \right)}
\newcommand{\EC}[2]{\mathbb{E} \left( #1 \| #2 \right)}

\newcommand{\supnorm}[1]{\| #1 \|_\infty}
\newcommand{\lpnorm}[2]{\| #1 \|_{\mathcal{L}_{#2} }}
\newcommand{\limN}{\underset{n \to \infty}{\lim}}
\newcommand{\plim}{\overset{p}{\longrightarrow}}

\newcommand{\Tn}{\mathcal{T}_{n}}
\newcommand{\Tnp}{\mathcal{D}_{n+1}}
\newcommand{\Tnls}[2]{
    \ifthenelse{\equal{#1}{n+1}}
    {\Tn^{\tilde{#2} \leftrightarrow #2}}
    {\Tnp^{\backslash #1, \tilde{#2} \leftrightarrow #2}}
}
\newcommand{\Tnlt}[2]{
    \ifthenelse{\equal{#1}{n+1}}
        {\Tn^{\backslash #2}}
        {\ifthenelse{\equal{#2}{n+1}}
            {\Tn^{\backslash #1}}
            {\Tnp^{\backslash #1,#2}}
        }
}

\newcommand{\Tnl}[1]{
    \ifthenelse{\equal{#1}{n+1}}
    {\Tn}
    {\Tnp^{\backslash #1}}
}

\newcommand{\Pfc}[2]{PS^{\text{fc}}_{n, #1}(#2)}
\newcommand{\PI}[1]{\Pfc{\alpha}{#1}}
\newcommand{\Pfcs}[2]{PS^{\text{sc}}_{n, #1}(#2)}
\newcommand{\PIs}[1]{\Pfcs{\alpha}{#1}}
\newcommand{\Pcc}[2]{PS^{\text{cc}}_{n, #1}(#2)}
\newcommand{\cm}[1]{\mathcal{C}_n(t_{#1},\Tnl{#1})}
\newcommand{\cms}[2]{\mathcal{C}_n(t_{#1},\Tnls{#1}{#2})}
\newcommand{\cmt}[2]{\mathcal{C}_{n-1}(t_{#1},\Tnlt{#1}{#2})}

\newcommand{\ldnamEnd}{L\'{e}vy gauge}
\newcommand{\ldname}{\ldnamEnd~}
\newcommand{\gam}{\delta}
\newcommand{\ld}[1]{\mathcal{L}_{#1}}
\newcommand{\LD}{\ld{\gam}(F,G)}
\newcommand{\Q}[2]{#2^{-1}(#1)}
\newcommand{\fatir}{\text{ for all } t \in \R}
\renewcommand{\epsilon}{\varepsilon}

\newcommand{\Op}[1]{\mathcal{O}_p\left(#1\right)}

\newenvironment{nalign}{
    \begin{equation}
    \begin{aligned}
}{
    \end{aligned}
    \end{equation}
    \ignorespacesafterend
}

\newlist{lemmenum}{enumerate}{1}
\setlist[lemmenum]{label=\alph*), ref=\thelemma~(\alph*)}
\crefalias{lemmenumi}{lemma} 

%% file: abstract.tex
\begin{abstract}
    Prediction sets based on full conformal prediction have seen an increasing interest in statistical learning due to their universal marginal coverage guarantees.
    However, practitioners have refrained from using it in applications for two reasons:
    Firstly, it comes at very high computational costs, exceeding even that of cross-validation. Secondly, an applicant is typically not interested in a marginal coverage guarantee which averages over all possible (but not available) training data sets, but rather in a guarantee conditional on the specific training data.
    
    This work tackles these problems by, firstly, showing that full conformal prediction sets are conditionally conservative given the training data if the conformity score is stochastically bounded and satisfies a stability condition. 
    Secondly, we propose an approximation for the full conformal prediction set that has asymptotically the same training conditional coverage as full conformal prediction under the stability assumption derived before, and can be computed more easily.
    Furthermore, we show that under the stability assumption, $n$-fold cross-conformal prediction also has the same asymptotic training conditional coverage guarantees as full conformal prediction.
    
    If the conformity score is defined as the out-of-sample prediction error, our approximation of the full conformal set coincides with the symmetrized Jackknife.
    We conclude that for this conformity score, if based on a stable prediction algorithm, full-conformal, $n$-fold cross-conformal, the Jackknife+, our approximation formula, and hence also the Jackknife, all yield the same asymptotic training conditional coverage guarantees.

\end{abstract}

%% file: paper.tex
\section{Introduction}
\subsection{Background}
In the past years, uncertainty quantification and predictive inference have seen an increased interest since in modern data science one is often confronted with a high-dimensional framework, for which many of the classic results established in statistics no longer hold. Prediction intervals based on fitted errors or even the bootstrap may fail if the number of explanatory variables is not negligible compared to the number of observations (cf. \cite{mammen1996empirical}, \cite{bickel1983bootstrapping}, \cite{elkaroui2018can}).
One possible approach is given by $k$-fold cross-validation, its special case, the Jackknife, and its variant CV+, which was recently introduced by \textcite{barber2021predictive}. 
However, these procedures may not even provide marginally valid prediction intervals in general: In fact, \textcite{barber2021predictive} showed the existence of non-stable algorithms such that the marginal coverage probability of the Jackknife and Jackknife+ intervals is essentially $0$ and $1-2\alpha$ for the Jackknife+ (where $1-\alpha$ was the target). 

An alternative is given by conformal prediction, which provides a universal marginal coverage guarantee as long as the data are exchangeable (cf. \cite{vovk2005algorithmic}). While the split conformal method lacks efficiency because the predictor is trained only on a subset of the data (cf. \cite{steinberger2020conditional} for a comparison with the Jackknife), a different picture arises for full conformal prediction, which uses all available training data in the fitting and calibration step.
However, practitioners have refrained from using the latter method for two reasons: Firstly, the prediction sets are computationally very expensive, even exceeding the Jackknife's computational costs. Secondly, a practitioner is typically interested in a conditional coverage guarantee given its training data rather than a marginal coverage guarantee, which averages over all possible (but unavailable) training data sets. 
\textcite{bian2022training} showed that full conformal prediction is not training conditionally conservative for all conformity scores.
Given this negative result, we might be willing to restrict our analysis of the training conditionally conservativeness of full conformal prediction sets to a large class of conformity scores of practical interest.

Another alternative is cross-conformal prediction, which typically comes with lower computational costs than full conformal prediction. However, due to its similarity to the Jackknife+ for one specific conformity score, the negative results for the coverage probabilities of the Jackknife+ immediately carry over to cross-conformal prediction (cf. Appendix B.2.1 in \cite{barber2021predictive}). It is therefore an open question whether cross-conformal prediction might be training conditionally conservative/valid for, at least, some typical conformity scores of interest.

\subsection{Contributions}
This work tackles the two significant challenges of full conformal prediction: Firstly, we show that prediction sets based on full conformal prediction are training conditionally conservative in large samples if the conformity score is stochastically bounded and satisfies a stability condition (cf. \Cref{thm:ma_asymptoticResult}). Secondly, we propose an approximation of full conformal prediction sets that satisfies the same asymptotic conditional coverage guarantees as the original version given the stability condition (cf. \Cref{thm:sc_equivalence}). This new variant has lower computational costs since the quantiles need not be recalculated for every $y \in \R$. Furthermore, we can also show that under some regularity assumptions, the expected length of the set difference between these two methods vanishes asymptotically (cf. \Cref{prop:dis_setDifference}). Thus, the two prediction sets are closely related despite their different computational costs.

We demonstrate its usefulness by applying it to two typical conformity scores:
Suppose the conformity score is a unimodal absolute in-sample prediction error. In that case, the shortcut formula can be approximated up to a deviation of $\epsilon$ by a bisection algorithm that needs $\mathcal{O}(\log_2(1/\epsilon))$ steps (cf. \Cref{alg:main} and \Cref{prop:app_algorithm}). Furthermore, we show that for some particular prediction errors, the shortcut formula needs no model refits at all, thereby having a computational efficiency comparable to sample splitting.
If the conformity score is the absolute out-of-sample prediction error, our approximation formula coincides with the Jackknife, which provides a connection between the Jackknife and full-conformal prediction (cf. \Cref{cor:jackknifeEquiv}).

Besides introducing a shortcut formula for full conformal prediction, this work comes with some additional results that might be interesting for their own sake.
We show that the actual training conditional coverage probability of full conformal prediction sets also avoids \emph{overshooting} the nominal level if the conformity score is stable and fulfills a continuity assumption (cf. \Cref{thm:ma_asymptoticValidity}).
Furthermore, we extend our results to $n$-fold cross-conformal prediction sets by showing that they have the same asymptotic training conditional coverage guarantees as full conformal if the underlying conformity score is stable, thereby providing another computationally attractive alternative to full conformal prediction (cf. \Cref{cor:app_CrossConfEquiFC}).
However, since the computational costs of our approximation formula are not larger than those of $n$-fold cross-conformal prediction in general, and substantially lower for some special cases, we consider the computational benefits of $n$-fold cross-conformal to be somewhat limited.

As a side result, we conclude that, if the conformity score is based on the out-of-sample prediction error of a stable algorithm, full-conformal prediction also provides the same asymptotic conditional coverage guarantees as $n$-fold cross-conformal prediction, the Jackknife, and the Jackknife+ (cf. \Cref{cor:summarizeEquivalences}). Furthermore, we give a weaker version of this result for the conformity score based on the in-sample prediction error of a stable algorithm under the additional assumption that the distribution of the feature vector is nonatomic (cf. \Cref{thm:app_fcBecomesJackknife}).

\subsection{Organization of the paper}
\Cref{sec:setting} introduces the setting and some definitions.
In \Cref{sec:fullConformal}, we define prediction sets based on full conformal and show that they are training conditionally conservative under a suitable stability condition on the conformity score, while \Cref{sec:shortcut} introduces the shortcut formula for full conformal prediction sets and proves its asymptotic equivalence to the original definition.
In \Cref{sec:discussion}, we present our results for the $n$-fold cross-conformal prediction and discuss the asymptotic equivalence of the Jackknife, the Jackknife+, $n$-fold cross-conformal prediction, and full conformal prediction in one special case.

Further results, including an extension of this connection to the conformity score based on the in-sample prediction error, can be found in \Cref{sec:furtherResults}.
We present an explicit algorithm for approximating the shortcut formula for unimodal conformity scores in \Cref{sec:algorithm}. 
Our proofs, given in \Cref{sec:proofs}, crucially rely on a recently introduced concept, the \ldnamEnd, which is presented in \Cref{sec:ld}.

\subsection{Notation}
For a set $S \subseteq \R^k$, we denote with $|S|$ its cardinality and $\mathds{1}_S: \R^k \to \R$ the indicator function on the set $S$. Furthermore, let $S^n = \otimes_{i=1}^n S$ denote its $n$-th cartesian product.
The expression $\lim_{x \searrow t}$ will denote the limit from above (in contrast to the usual limit $\lim_{x \to t}$). 
For a distribution function $F$, we define $F(x-):= \lim_{\delta \searrow 0}F(x-\delta)$ to be the limit from the left at the point $x \in \R$. Furthermore, for a vector $x$, we write $x^T$ for the transposed vector and denote the $i$-th canonical basis vector with $e_i$.

For a sequence of random variables $X_n$, we will abbreviate convergence in probability to a random variable $X$ by $X_n \plim X$. 
Furthermore, for a sequence $(v_n)_{n \in \N}$ we will write 
$X_n \sim \Op{v_n}$ if the sequence $(X_n v_n^{-1})_{n \in \N}$ is stochastically bounded. 
For a measurable function $f: \R \to \R$, we will denote its supremum norm with $\supnorm{f} := \sup_{x \in \R} |f(x)|$, the
$\mathcal{L}_p$ norm with $\lpnorm{f}{p} := ( \int_{\R}|f(x)|^p d \lambda(x) )^{1/p}$ for $p \in [1, \infty)$, where $\lambda$ denotes the Lebesgue-measure.

\section{Setting}\label{sec:setting}
\subsection{The data}
Since we want to embed our finite sample results in an asymptotic setting that includes the high-dimensional case, we allow our model to depend on $n$ and pose the following assumptions:

For every $n \in \N$, let $P_n$ be some regular Borel probability measure on the space $\mathcal{Z}_n = \mathcal{Y}_n \times \mathcal{X}_n$ where $\mathcal{Y}_n \subseteq \R$ is a (possibly unbounded) interval and $\mathcal{X}_n$ is some space. For simplicity, we will assume that $\mathcal{X}_n$ is a $p_n$-dimensional Euclidean space with $p_n \in \N$, but our arguments can easily be generalized to other spaces.
For each $n \in \N$ the data $t^{(n)}_i = (y^{(n)}_i, x^{(n)}_i)$ with $1 \leq i \leq n+1$ are independent and identically distributed according to the measure $P_n$. In particular, $y^{(n)}_i$ is a $\mathcal{Y}_n$-valued random variable for each $n \in \N$ and $1 \leq i \leq n+1$ whose distribution may depend on $n$. Since the space $\mathcal{X}_n$ of $x^{(n)}_i$ may change with the sample size, our setting also includes the high-dimensional case $\mathcal{X}_n = \R^{p_n}$ for some sequence of dimensions $p_n$ fulfilling $\limN p_n/n > 0$. Furthermore, we will denote the marginal distributions of $x_{n+1}$ and $y_{n+1}$ with $P_n^x$ and $P_n^y$, respectively.

We then define the \emph{training} data $\Tn^{(n)}$ as $(t_1^{(n)}, \ldots, t_n^{(n)})$ and the full data set $\Tnp^{(n)}$ including the new feature-response pair $t^{(n)}_{n+1}$ as $(t_1^{(n)}, \ldots, t^{(n)}_{n+1})$. We emphasize that in the setting we are interested in, only the feature vector $x^{(n)}_{n+1}$ is known, and the future response $y^{(n)}_{n+1}$ is not available. 
In our proofs and statements, we will often use the shrunken set $\Tnp^{(n) \backslash i}$
with $1 \leq i \leq n+1$ which is defined as $\Tnp^{(n) \backslash i} = \Tnp^{(n)} \backslash t^{(n)}_i$. In particular, $\Tnp^{(n) \backslash n+1}$ coincides with  $\Tn^{(n)}$.
Furthermore, we will write $\Tn^{(n) \backslash i}$ as a shorthand for $\Tn^{(n)} \backslash t^{(n)}_i$.
To ease the notation, we will suppress the dependence of the distribution $\mathcal{P}_n$ on $n$ whenever it is clear from the context and write $t_i$ for $t_i^{(n)}$, $y_i$ for $y_i^{(n)}$, $\Tn$ for $\Tn^{(n)}$, $\Tnp$ for $\Tnp^{(n)}$, $\Tnl{i}$ for $\Tnp^{(n) \backslash i}$.

\subsection{The goal}
The \emph{first goal} of this paper is to create a prediction set $PS$ for $y_{n+1}$ given its regressor $x_{n+1}$ and the training data $\Tn = (t_1, \ldots, t_n)$.
Ideally, we want our prediction set to be training conditionally valid in large samples, that is, 
\begin{align}\label{eq:set_asyExactCovProb}
    \left| \PC{y_{n+1} \notin PS_\alpha}{\Tn} - \alpha \right| \plim 0 \text{ for all } \alpha \in [0,1].
\end{align}
However, even for large samples, the coverage probability of any non-randomized prediction set may not be close to its target level for every $\alpha \in [0,1]$ if the distribution of $t_{n+1}$ is non-continuous. For example, if $t_{n+1}$ is Dirac-distributed, no other coverage probability than $0$ or $1$ is possible for non-randomized prediction intervals. 
Therefore, we are willing to accept a weaker guarantee for general, possibly non-continuous, distributions of the form
\begin{align}\label{eq:set_asyConsCovProb}
    \P \left[ \PC{y_{n+1} \notin PS_\alpha}{\Tn} > \alpha + \epsilon_n \right] \leq \epsilon_n
\end{align}
where $(\epsilon_n)_{n \in \N}$ is a null-sequence,
and a guarantee of the form \eqref{eq:set_asyExactCovProb} for distributions fulfilling a continuity condition.

We emphasize that the coverage guarantees should hold conditionally on the training data rather than the new feature vector $x_{n+1}$.\footnote{For the so-called object conditional validity of prediction sets and their limitation, we refer to \cite{vovk2012conditional} and \cite{barber2021predictive}.} In the following, the terms \emph{conditional validity} and \emph{conditional conservativeness} therefore always refer to conditioning on the training data.

Our \emph{second goal} is to find a prediction set $PS$ which not only fulfills \cref{eq:set_asyConsCovProb} (or even  \cref{eq:set_asyExactCovProb} in the continuous case), but can also be computed efficiently in practice.

\section{Full conformal prediction}\label{sec:fullConformal}
The prediction sets we are considering are based on full conformal prediction. To define them, we need the concept of a conformity score and augmented data.

For each $n,m \in \N$ let the conformity score $\mathcal{C}^{(n)}_m: \mathcal{Z}_n \times \mathcal{Z}_n^m \to \R$ be a measurable function which is symmetric in its second argument, that is,
for each $t_{m+1} \in \mathcal{Z}_n$, $(t_1, \ldots, t_m) \in \mathcal{Z}_n^m$ 
and every permutation $\pi: \{1, \ldots, m\} \to \{1, \ldots, m\}$ we have
$\mathcal{C}^{(n)}_m(t_{m+1}, (t_1, \ldots, t_m)) = f(t_{m+1}, (t_{\pi(1)}, \ldots, t_{\pi(m)}))$. 
To ease the notation, we abbreviate the family $(\mathcal{C}^{(n)}_m)_{m,n \in \N}$ with $\mathfrak{C}$ and drop the index $n$ whenever it is clear from the context.

Note that we here use the original definition of the conformity score $\mathcal{C}_m(t_{m+1}, (t_1, \ldots, t_m))$ rather than the more common notation $\tilde{\mathcal{C}}_m(t_{m+1}, (t_1, \ldots, t_m, t_{m+1}))$. Indeed, these definitions can be used interchangeably and do not influence the validity of our proofs. We refer to Chapter~$2.9.3$ of \textcite{vovk2022algorithmic} for a discussion.\footnote{
    To match the classical definition of full conformal prediction, we can replace $\mathcal{C}_n((y, x_{n+1}),\Tn)$ and $\mathcal{C}_{n+1}(t_i,\Tnp^y \backslash t_i)$ by $\tilde{\mathcal{C}}_{n+1}((y, x_{n+1}),\Tnp^y)$ and $\tilde{\mathcal{C}}_n(t_i,\Tnp^y)$, respectively. The difference between the two definitions is negligible for our purpose, and we may use any of them to get the same results.
}

For a $y \in \mathcal{Y}_n$ we define the augmented data $\Tnp^y$ as $((y_1, x_1), \ldots, (y_n, x_n), (y, x_{n+1}))$ and a distribution
function $\hat{F}_y$ as follows:
\begin{align*}
    \hat{F}_y(s) = \dfrac{\mathds{1}\{ \mathcal{C}_n((y, x_{n+1}),\Tn) \leq s\}}{n+1}
    + \dfrac{1}{n+1} \sum_{i=1}^n \mathds{1}\{ \mathcal{C}_n(t_i,\Tnp^y \backslash t_i) \leq s\} \text{ for all } s \in \R.
\end{align*}

Next, we will define quantiles of distribution functions. We extend the definition to all $\alpha \in \R$ to avoid case distinctions in our proofs and statements.
\begin{definition}[Quantiles]\label{def:ld_quantiles}
    Let $F: \R \to [0,1]$ be a cumulative distribution function and $\alpha \in \R$. 
    We then define the $\alpha$-quantiles $\Q{\alpha}{F}$ of $F$ (where $\alpha$ is extended to the whole real line) as follows:
    \begin{align}\label{eq:ld_quantileDefinition}
        \Q{\alpha}{F} := \inf\{x \in \R: F(x) \geq \alpha\} \in [-\infty, +\infty],
    \end{align}
    where we define the infimum of the empty set $\inf\{x \in \R: x \in \emptyset\}$ as $+\infty$.
    In particular, we have $\Q{\alpha}{F} = - \infty$ whenever $\alpha \leq 0$ and $\Q{\alpha}{F} = +\infty$ whenever $\alpha > 1$.
\end{definition}

Denoting the the limit from the left of the function $F$ at the point $t$ with $F(t-)$, 
we immediately conclude from \Cref{def:ld_quantiles}
\begin{align}\label{eq:ld_alphaQuantilesExceedAlpha}
    F(\Q{\alpha}{F}-) \leq \alpha \leq F(\Q{\alpha}{F})
\end{align}
whenever $\Q{\alpha}{F}$ is finite.

Then, for any $\gam \in \R$, the $\gam$-distorted prediction set based on full conformal prediction with nominal coverage probability of $1-\alpha \in [0,1]$ is defined as
\begin{align}\label{eq:set_defineFullConfSets}
    \PI{\gam} = \{y \in \mathcal{Y}_n: \mathcal{C}_n((y, x_{n+1}), \Tn) \leq \Q{1-\alpha}{\hat{F}_y} + \gam\},
\end{align}
where we use the convention that for all $y \in \mathcal{Y}_n$ the inequality $y \leq +\infty$ is true while we have $y \nleq -\infty$. To avoid case distinctions, we extend this definition to all $\alpha \in \R$. If $\gam > 0$, we call $\PI{\gam}$ the $\gam$-\emph{inflated} full conformal prediction sets.

\begin{example}
    For $n,m \in \N$ we call a measurable function $\mathcal{A}_{n,m}: \mathcal{X}_n \times \mathcal{Z}_n^m \to \mathcal{Y}_n$ a prediction algorithm that is trained on a data set of size $m$ and evaluated at a feature vector in $\mathcal{X}_n$. Being somewhat sloppy, we will sometimes also call the family $\mathcal{A}:= (A_{n,m})_{n,m \in \N}$ a prediction algorithm.
    
    Then, a typical conformity score used in literature (cf. \cite{Lei2018DistributionfreePI}, \cite{barber2021predictive}, \cite{bian2022training}, \cite{liang2025algorithmic}) is the (absolute) in-sample prediction error $\mathcal{C}_n^{in}$:
    \begin{align}\label{eq:set_isPe}
        \mathcal{C}_n^{in}((y, x),\Tn) = |y - \mathcal{A}_{n, n+1}(x, \Tn \cup (y, x))|.
    \end{align}
    Another possibility is to use the (absolute) out-of-sample prediction error $\mathcal{C}_n^{out}$ as a conformity score, that is, 
    \begin{align}\label{eq:set_oosPe}
       \mathcal{C}_n^{out}((y, x),\Tn) = |y - \mathcal{A}_{n, n}(x, \Tn)|.
    \end{align}
\end{example}

The following proposition is a well-known result for the case $\gam = 0$ (see, e.g., Proposition $2.3$ in \textcite{vovk2022algorithmic}) and naturally extends to inflated prediction sets.
\begin{proposition}\label{lem:set_marginalCoverage}
    For any family of conformity scores $\mathfrak{C} = (\mathcal{C}_n)_{n \in \N}$,
    and any $\gam \geq 0$, the corresponding full conformal prediction sets defined in \cref{eq:set_defineFullConfSets} are marginally conservative, that is,
    for all $\alpha \in [0,1]$ we have $\P(y_{n+1} \notin \PI{\gam}) \leq \alpha$.
\end{proposition}

The drawback of marginal coverage guarantees is that they average over all possible, but not available, training data. Thus, a coverage guarantee conditional on the training data is more desirable in practice. \cite{bian2022training} give an example of a conformity score, where the actual training conditional coverage probability is essentially $0$ for a set of training data with measure $\alpha$ (cf. Theorem~$2$ therein).
Given this negative result, we restrict our analysis of the actual conditional coverage probability of full conformal prediction sets to the class of so-called \emph{stable} conformity scores: 

\begin{definition}[Stable conformity score]\label{def:Fc_stableConformityScore}
    We call a family of conformity scores $\mathfrak{C} = (\mathcal{C}_n)_{n \in \N}$ stable, if
    \begin{align*}
        |\cm{n+1} - \cmt{n+1}{n}| \plim 0.
    \end{align*}
\end{definition}

Note that the definition of stability is an interplay between $\mathfrak{C}$ and the sequence of distributions $(\mathcal{P}_n)_{n \in \N}$.

\begin{theorem}[Asymptotically conditionally conservative prediction sets]\label{thm:ma_asymptoticResult}
    Assume $\mathfrak{C}$ is stable and
    $(\mathcal{C}_n(t_{n+1}, \Tn))_{n \in \N}$ is stochastically bounded.
    Then, $\gam$-inflated full conformal prediction sets are asymptotically uniformly conditionally conservative, that is, for all $\gam > 0$ we have
    \begin{align*}
        \limN \P \left( \sup_{\alpha \in [0,1]} 
            \PC{y_{n+1} \notin \PI{\gam}}{\Tn} - \alpha \geq \epsilon \right) = 0
            \text{ for all } \epsilon > 0.
    \end{align*}
\end{theorem}

\Cref{thm:ma_asymptoticResult} shows that the conditional coverage probability of $\gam$-inflated full conformal prediction sets does not undershoot its nominal level asymptotically.
In the special case where the conformity score is defined as the in-sample prediction error $\mathcal{C}_n^{in}$, a slightly weaker result has appeared in the literature (cf. Theorem~3.7 in \cite{liang2025algorithmic}).\footnote{See \Cref{sub:dis_comparison} for a comparison of our results.}

We emphasize that \Cref{thm:ma_asymptoticResult} is based on coverage guarantees for fixed $n \in \N$ (cf. \Cref{thm:ma_finiteSampleResult}). From this finite sample result, it can also be shown that the conclusion of \Cref{thm:ma_asymptoticResult} remains true if we replace its conditions by the assumption that the $\limN \E|\cm{n+1} - \cmt{n+1}{n}| = 0$ and $\limN \frac{1}{n} \E(\cm{n+1}) = 0$.

\begin{remark}[Data-driven choice of the nominal coverage probability]
    Since the statement of \Cref{thm:ma_asymptoticResult} is uniform in $\alpha \in [0,1]$ \emph{conditional on the training data}, it allows to choose $\alpha$ data-dependent. In particular, we have
    \begin{align*}
        \limN \P \left( \PC{y_{n+1} \notin \Pfc{\alpha(\Tn)}{\gam}}{\Tn} \geq \alpha(\Tn) + \epsilon \right) = 0
    \end{align*}
    for every $\alpha(\Tn)$ being measurable with respect to the training data $\Tn$.
    In particular, this allows the practitioner to choose $\alpha$ \emph{after} observing the data.
\end{remark}

However, the major drawback of \Cref{thm:ma_asymptoticResult} is that we can only bound the conditional coverage probability from below and, additionally, it is only applicable for slightly inflated prediction sets. If we restrict our analysis to a continuous case, we get an even stronger statement proving the asymptotic validity of \emph{non-inflated} prediction sets.

\begin{definition}[Continuous Case Assumption]\label{def:set_continuity}
    We say the Continuous Case Assumption (CC) is fulfilled if for every $n \in \N$ and almost every training data $\Tn$ the
    distribution of the conformity score $\cm{n+1}$ conditional on $\Tn$ is absolutely continuous with density $f_{\Tn}$, and the sequence of random variables $(\supnorm{f_{\Tn}})_{n \in \N}$ is stochastically bounded.
\end{definition}

\begin{example}
    Assume the conformity score $\cm{n+1}$ is defined as the out-of-sample prediction error $\mathcal{C}_n^{out}$ as in \cref{eq:set_oosPe}. Then, \Cref{def:set_continuity} is fulfilled whenever for every $n \in \N$ the distribution of $y_{n+1}$ conditional on $x_{n+1} = x$ is absolutely continuous for almost all $x$ and the supremum norm of its density is bounded over all  $x \in \mathcal{X}_n$ and $n \in \N$.\footnote{
        A similar continuity condition has also appeared for the Jackknife and the Jackknife+ in the literature (cf. \cite{liang2025algorithmic} and \cite{amann2025UQ}).
    }
    In particular, this continuity condition with $\mathcal{C}_n^{out}$ is fulfilled if $y_{n+1} = f(x_{n+1}) + u_{n+1}$ for a measurable function $f$ and an absolutely continuous random variable $u_{n+1}$ that is independent of $x_{n+1}$ and whose density's supremum norm $\supnorm{f^{(n)}_u}$ is bounded over $n \in \N$.
\end{example}

Equipped with this definition, we can state our next result.

\begin{theorem}[Asymptotically valid prediction sets under continuity]\label{thm:ma_asymptoticValidity}
    Assume $\mathfrak{C}$ is stable
    and $(\mathcal{C}_n(t_{n+1}, \Tn))_{n \in \N}$ is stochastically bounded.
    If the continuous case assumption CC is fulfilled, then the non-inflated prediction sets based on full conformal prediction are asymptotically uniformly conditionally valid, that is,
    \begin{align*}
        \limN \E\left( \sup_{\alpha \in [0,1]} \left| \PC{y_{n+1} \notin \PI{0}}{\Tn} - \alpha \right| \right) = 0.
    \end{align*}
\end{theorem}

\Cref{thm:ma_asymptoticValidity} shows that in the continuous case, the actual conditional coverage probability of non-inflated full conformal prediction sets is also not overshooting their nominal level in large samples. 
In particular, for every fixed $\alpha \in [0,1]$, their conditional coverage probability converges to its nominal level in probability. 
Loosely speaking, \Cref{thm:ma_asymptoticValidity} suggests that asymptotically full conformal prediction sets are not overly large.\footnote{
    Of course, \Cref{thm:ma_asymptoticValidity} does not exclude the existence of another prediction set that has a significantly smaller length but the same coverage probability.
}

By proving the training conditional validity of the full conformal prediction sets, 
we extend the results of \textcite{liang2025algorithmic} to the case of general stable conformity scores and, additionally, show that the prediction sets are not only conditionally conservative in large samples but also do not overshoot the nominal level under a continuity assumption (cf. \Cref{thm:ma_asymptoticValidity}). In particular, the latter statement has been an open question even for the special case treated in \textcite{liang2025algorithmic}.

\section{A shortcut formula}
\label{sec:shortcut}
As shown in \Cref{sec:fullConformal}, full conformal prediction sets are conditionally conservative in large samples if based on stable and stochastically bounded conformity scores. Given the negative result of \textcite{bian2022training} (cf. Theorem~$2$ therein), we cannot hope that these conditional coverage guarantees can be extended to every arbitrary, potentially non-stable, conformity score. Thus, we will restrict our further analysis to stable conformity scores.

Nevertheless, \Cref{sec:fullConformal} does not solve the main problem arising in practice: Since the model has to be refitted for every $y \in \mathcal{Y}_n$, it comes with high computational costs.
As we will see in \Cref{sub:unimodal} and \Cref{sec:connection}, the main challenge for several typical conformity scores lies in calculating $\hat{F}_{y}$ while the calculation of the conformity score $\mathcal{C}_n((y, x_{n+1}),\Tn)$ can often be done efficiently by exploiting the structure of the conformity score. In such a scenario, we propose the following modification of full-conformal prediction sets:
\begin{definition}\label{def:sc_Pis}
    For a $\gam \in \R$ we define the $\gam$-distorted shortcut prediction sets based on the conformity score $\mathcal{C}_n$ as
    \begin{align*}
        \PIs{\gam} = \{y \in \mathcal{Y}_n: \mathcal{C}_{n}((y, x_{n+1}), \Tn) \leq \Q{1-\alpha}{\hat{G}} + \gam\},
    \end{align*}
    where $\hat{G}$ is the empirical distribution function (ecdf) of the non-augmented conformity scores $(\mathcal{C}_{n-1}(t_i, \Tnlt{n+1}{i}))_{i=1}^n$, i.e.,
    \begin{align*}
        \hat{G}(t) = \dfrac{1}{n}\sum_{i=1}^n \mathds{1}_{[\mathcal{C}_{n-1}(t_i, \Tnlt{n+1}{i}), \infty)}(t).
    \end{align*}
\end{definition}

The main advantage of using the modification lies in the fact that the function $\hat{G}$ is independent of the point $y$ and therefore needs to be computed only once. 
We stress that now $\PIs{\gam}$ violates the invariance principle of full conformal and therefore might not be marginally valid for \emph{any} distribution and conformity score. However, in practice this marginal coverage guarantee might be useless for the statistician if the prediction set is not conditionally valid (at least in large samples) as in such a setting the actual coverage probability for the available training data can essentially be $0$ (and the probability of getting such a training set is approximately $\alpha > 0$). 

We showed in \Cref{sec:fullConformal} that stable conformity scores come with a conditional coverage guarantee in large samples (under the minor additional assumption of the stochastic boundedness of the conformity score). 
In such a setting, the shortcut formula of \Cref{def:sc_Pis} possesses the same conditional coverage guarantee in large samples as the following result shows.
\begin{theorem}[Asymptotic equivalence between the shortcut and the original full conformal sets]\label{thm:sc_equivalence}
    Assume $\mathfrak{C}$ is stable.
    \begin{enumerate}[label=(\roman*)]
        \item \emph{General case:} Then, the following statements are equivalent:
            \begin{enumerate}[label=G\arabic*)]
                \item\label{it:sc_Gfc} For all $\gam > 0$, the $\gam$-inflated full conformal prediction sets are asymptotically uniformly conditionally conservative, i.e.,
                    \begin{align*}
                        \limN \P\left( \sup_{\alpha \in [0,1]} \PC{y_{n+1} \notin \PI{\gam}}{\Tn} - \alpha 
                        \geq \epsilon  \right) = 0 \text{ for all } \epsilon > 0.
                    \end{align*}
                \item\label{it:sc_Gsc} For all $\gam > 0$ the $\gam$-inflated shortcut sets $\PIs{\gam}$ as in \Cref{def:sc_Pis} are asymptotically uniformly conditionally conservative, i.e.,
                    \begin{align*}
                        \limN \P\left( \sup_{\alpha \in [0,1]} \PC{y_{n+1} \notin \PIs{\gam}}{\Tn} - \alpha 
                        \geq \epsilon \right) = 0 \text{ for all } \epsilon > 0.
                    \end{align*}
            \end{enumerate}
        \item \emph{Continuous case:} If, additionally, the continuous case assumption CC is fulfilled, then the following statements are equivalent:
            \begin{enumerate}[label=C\arabic*)]
                \item\label{it:sc_Cfc} The original, non-inflated full conformal prediction sets are 
                    asymptotically uniformly conditionally valid, i.e.,
                    \begin{align*}
                        \limN \E\left( \sup_{\alpha \in [0,1]} \left| \PC{y_{n+1} \notin \PI{0}}{\Tn} - \alpha \right| \right) = 0.
                    \end{align*}
                \item\label{it:sc_Csc} The non-inflated shortcut sets $\PIs{0}$ are 
                    asymptotically uniformly conditionally valid, i.e.,
                    \begin{align*}
                        \limN \E\left( \sup_{\alpha \in [0,1]} \left| \PC{y_{n+1} \notin \PIs{0}}{\Tn} - \alpha \right| \right) = 0.
                    \end{align*}
            \end{enumerate}
    \end{enumerate}
\end{theorem}

We will abbreviate the asymptotic equivalence results of \Cref{thm:sc_equivalence} by simply writing $\Pfc{\cdot}{\cdot} \leftrightharpoons \Pfcs{\cdot}{\cdot}$ and use this notation in \Cref{sec:discussion}. In fact, a slightly stronger equivalence statement than \Cref{thm:sc_equivalence} can be made, which, under some regularity conditions, imply that the expected length of the set difference between $\PI{0}$ and $\PIs{0}$ vanishes asymptotically (cf. \Cref{prop:dis_setDifference} for the full statement).

Combining \Cref{thm:ma_asymptoticResult} and \Cref{thm:ma_asymptoticValidity} with \Cref{thm:sc_equivalence} immediately gives the following result:
\begin{corollary}[Asymptotic conservativeness/validity of the shortcut sets]\label{cor:validityOfShortcut}
    Assume $\mathfrak{C}$ is stable and $(\mathcal{C}_n(t_{n+1}, \Tn))_{n \in \N}$ is stochastically bounded.
    \begin{enumerate}[label=(\roman*)]
        \item \emph{General case:} Then, for any $\gam > 0$ the $\gam$-inflated shortcut prediction sets
            are asymptotically uniformly conditionally conservative, i.e.,
            \begin{align*}
                \limN \P\left( \sup_{\alpha \in [0,1]} \PC{y_{n+1} \notin \PIs{\gam}}{\Tn} - \alpha 
                \geq \epsilon \right) = 0 \text{ for all } \epsilon > 0.
            \end{align*}
        \item \emph{Continuous case:} If, additionally, the continuous case assumption CC is fulfilled, then the non-inflated shortcut prediction sets are
            asymptotically uniformly conditionally valid, i.e.,
            \begin{align*}
                \limN \E\left( \sup_{\alpha \in [0,1]} \left| \PC{y_{n+1} \notin \PIs{0}}{\Tn} - \alpha \right| \right) = 0.
            \end{align*}
    \end{enumerate}
\end{corollary}

To sum it up, we have shown that for stable conformity scores, the asymptotic conditional coverage guarantees of the shortcut formula coincide with the ones for the original full conformal prediction sets. Thus, in practice, one can use the shortcut formula for stable conformity scores due to its computational advantages compared to full conformal prediction without losing the conditional coverage guarantees in large samples.

In particular, \Cref{cor:validityOfShortcut} shows that for all $\gam > 0$ the $\gam$-inflated shortcut sets based on stable conformity scores $\mathfrak{C}$ are asymptotically marginally conservative, that is, for all $\alpha \in [0,1]$ we have
\begin{align*}
    \limsup_{n \to \infty} \P(y_{n+1} \notin \PIs{\gam}) \leq \alpha.
\end{align*}
Furthermore, in the continuous case, this statement remains true for the non-inflated shortcut sets.

In the following, we demonstrate the computational advantages of the shortcut formula for several typical conformity scores.

\subsection{In-sample prediction error}
We start with the conformity score $\mathcal{C}^{in}_n$ defined as the absolute in-sample prediction error based on a prediction algorithm $\mathcal{A}$ and assume for simplicity $\mathcal{Y}_n = \R$  throughout this section.
Then, $\hat{G}$ is the ecdf of the \emph{fitted} errors, and therefore $\Q{1-\alpha}{\hat{G}}$ is simply the $\lceil(1-\alpha)n\rceil$-th smallest fitted error.
In this scenario, the quantile $\Q{1-\alpha}{\hat{G}}$ can be computed efficiently, since the model has to be fitted only once. Thus, it remains to calculate the in-sample prediction error $|y - \mathcal{A}(x_{n+1}, \Tn \cup (y, x_{n+1}))|$ for different values of $y \in \mathcal{Y}_n$.

\subsubsection{Affine predictors}
Assume the prediction algorithm $\mathcal{A}$ is affine in the feature vector $(y_1, \ldots, y_n, y)$, that is, $\mathcal{A}(x_{n+1}, \Tnp^y) = f(x_1, \ldots, x_{n+1})^T (y_1, \ldots, y_n, y) + g(x_1, \ldots, x_{n+1}),$
for measurable functions $f: \R^{(n+1)p} \to \R^{n+1}$ and $g: \R^{(n+1)p} \to \R$.
Then, our shortcut formula can be computed without any refitting:
To see this, note that the conformity score can be rewritten as
\begin{align*}
    \mathcal{C}_n^{in}((y, x_{n+1}), \Tn) = 
    |y\underbrace{(1 - f(x_1, \ldots, x_{n+1})^T e_{n+1})}_{=:a(x_1, \ldots, x_{n+1})} - \underbrace{(y_1, \ldots, y_n, 0)^T f(x_1, \ldots, x_{n+1})}_{=:b(x_{n+1}, \Tn)}|.
\end{align*}
We then have
\begin{align*}
    \PIs{\gam} = \begin{cases}
            \left[b(x_{n+1}, \Tn) \pm \dfrac{\Q{1-\alpha}{\hat{G}} + \gam}{a(x_{n+1}, \Tn)}\right] & \text{if } a(x_{n+1}, \Tn) \neq 0\\
            \mathcal{Y}_n & \!\begin{aligned} & \text{if } a(x_{n+1}, \Tn) = 0 \text{ and }\\ & |b(x_{n+1}, \Tn)| \leq \Q{1-\alpha}{\hat{G}} + \gam \end{aligned} \\
            \emptyset & \text{else.}
    \end{cases}
\end{align*}
Since the model need not be refitted for the calculation of $\PIs{\gam}$, it comes with similar computational costs as sample splitting.

\begin{example}
    Typical examples of predictors that are linear in the feature vector are the Ridge and the OLS.
\end{example}

\subsubsection{kNN}
For the $k$-th nearest neighbor algorithm with $k \geq 2$, we can do a similar trick:
If the distribution of $x_{n+1}$ is nonatomic, then the in-sample prediction error of the $k$-th nearest neighbor prediction coincides with the (scaled) out-of-sample prediction error of the $(k-1)$-th nearest neighbor prediction almost surely:
\begin{align*}
    \left|y - \mathcal{A}^{kNN}(x_{n+1}, \Tnp^y)\right|
    = \frac{k-1}{k} \left|y - \mathcal{A}^{(k-1)NN}(x_{n+1}, \Tn)\right|,
\end{align*}
where $\mathcal{A}^{(k-1)NN}(x_{n+1}, \Tn)$ denotes the $(k-1)$-th nearest neighbor (out-of-sample) prediction for $x_{n+1}$ based on $\Tn$.
Thus, we have 
\begin{align*}
    \PIs{\gam} = \left[\mathcal{A}^{(k-1)NN}(x_{n+1}, \Tn) \pm \frac{k}{k-1} \left(\Q{1-\alpha}{\hat{G}} + \gam \right) \right] \text{ a.s.}
\end{align*}
Furthermore, note that the conformity score is stable whenever $\limN \frac{k}{n} = 0$.

\subsubsection{Unimodal conformity scores}\label{sub:unimodal}
This section demonstrates the usefulness of the new shortcut formula by providing an explicit, exponentially fast algorithm for unimodal conformity scores, in the sense that it can compute $\PIs{\gam}$ up to a deviation of $\epsilon$ in $\mathcal{O}(\log_2(1/\epsilon))$ steps.

\begin{definition}
    We call a conformity score unimodal, if for almost all $x_{n+1}$ and training data $\Tn$ there exists a point $m = m(x_{n+1}, \Tn) \in \R$ such that the function $y \mapsto \mathcal{C}((y, x_{n+1}), \Tn)$ is strictly decreasing on $(-\infty, m]$ and strictly increasing on $[m, \infty)$.
\end{definition}

\begin{example}
    A conformity score is unimodal whenever the function $y \mapsto \mathcal{C}((y, x_{n+1}), \Tn)$ is strictly convex and fulfills $\lim_{|y| \to \infty} \mathcal{C}((y, x_{n+1}), \Tn) = \infty$ for almost all $x_{n+1}$ and $\Tn$. 
\end{example}

The unimodality condition not only ensures that the prediction set of the shortcut method is a (possibly unbounded) interval, but more importantly, allows an exponentially fast approximation of it with a bisection algorithm. Starting with one point $m$ inside, one point $L$ left (but outside), and one point $U$ right (but outside) the prediction set allows to approximate the interval exponentially fast by halving the space between the inner point and the outer point in each step (see \Cref{alg:bisection} in \Cref{sec:algorithm}). However, this approach assumes that one can find these three points $L, m, U$. \Cref{alg:main} in \Cref{sec:algorithm} provides a procedure that includes this preliminary step without losing its computational efficiency (cf. \Cref{prop:app_algorithm} for a precise statement).

\subsection{Out-of-sample conformity score}\label{sec:connection}
The shortcut formula of \cref{def:sc_Pis} not only provides a fast approximation for unimodal conformity scores, but also shows a connection to the Jackknife: 
To see this, we consider the conformity score $\mathcal{C}_n^{out}$ based on the absolute out-of-sample prediction error of an algorithm $\mathcal{A}$.
In this case, the values $\mathcal{C}(t_i, \Tnlt{n+1}{i})$ coincide with the absolute value of the leave-one-out residuals $|\hat{u}_i| := |y_i - \mathcal{A}(x_i, \Tnlt{n+1}{i})|$ for $1 \leq i \leq n$.
Thus, $\Q{1-\alpha}{\hat{G}}$ is just the $(1-\alpha)$-quantile of the ecdf of the absolute leave-one-out residuals. To sum it up, in this special case, the shortcut formula of \Cref{def:sc_Pis} can be rewritten as
\begin{align*}
    \PIs{\gam} &= \{y \in \mathcal{Y}_n: |y - \mathcal{A}(x_{n+1}, \Tn)| \leq \Q{1-\alpha}{\hat{G}} + \gam\}\\
    &= \mathcal{A}(x_{n+1}, \Tn) + [-\Q{1-\alpha}{\hat{G}} - \gam, \Q{1-\alpha}{\hat{G}} + \gam].
\end{align*}
This is just the formula for the $\gam$-inflated prediction interval of the symmetrized Jackknife (cf. \cite{barber2021predictive}, \cite{amann2025UQ}).
Before summarizing these findings, we need the definition of asymptotic out-of-sample stable prediction algorithms.

\begin{definition}\label{def:app_AlgStability}
    We call a prediction algorithm $\mathcal{A}$ that is symmetric with respect to a permutation of the training data\footnote{
        A prediction algorithm $\mathcal{A}$ is called symmetric if it does not change by a permutation of the training data, that is, if for all $n \in \N$, all $x \in \mathcal{X}_n$, all training data $(t_1, \ldots, t_n)$ and every permutation $\pi: \{1, \ldots, n\}$ it fulfills $\mathcal{A}(x, (t_1, \ldots, t_n)) = \mathcal{A}(x, (t_{\pi(1)}, \ldots, t_{\pi(n)}))$.}  
    asymptotically out-of-sample stable (cf. \cite{amann2025UQ}) if it fulfills
    \begin{align*}
        \mathcal{A}_{n,n}(x_{n+1}, \Tn) - \mathcal{A}_{n,n-1}(x_{n+1}, \Tnlt{n+1}{n}) \plim 0.
    \end{align*}
\end{definition}

\begin{corollary}\label{cor:jackknifeEquiv}
    For the conformity score $\mathcal{C}_n^{out}$, the shortcut formula coincides with the symmetric Jackknife.
    If the symmetric prediction algorithm used for the Jackknife is asymptotically out-of-sample stable as in \Cref{def:app_AlgStability}, 
    then $\gam$-inflated prediction intervals based on the symmetrized Jackknife are asymptotically conditionally equivalent to full conformal prediction sets based on the conformity score $\mathcal{C}_n^{out}$ in the sense of \Cref{thm:sc_equivalence}. 
\end{corollary}

\begin{remark}
    Assume the conditions of \Cref{cor:jackknifeEquiv} hold and the prediction error $y_{n+1} - \mathcal{A}_{n,n}(x_{n+1}, \Tn)$ is bounded in probability.
    Then the assumptions of \Cref{cor:validityOfShortcut} are fulfilled. Therefore, the symmetrized Jackknife is asymptotically uniformly conditionally conservative in the general case and asymptotically uniformly conditionally valid in the continuous case. This result follows from the fact that $\PIs{\gam}$ coincides with the symmetrized Jackknife prediction intervals and is in line with the literature (cf. Proposition~B.1 of \cite{amann2025UQ}).
\end{remark}

\section{Discussion}\label{sec:discussion}
\subsection{Cross-conformal prediction}\label{sub:app_CrossConformal}
Another computationally beneficial alternative to full conformal prediction is cross-conformal prediction. We here investigate its conditional coverage probability in large samples.

\begin{definition}[$n$-fold cross-conformal prediction sets]\label{def:app_crossConformal}
    For $\gam \geq 0$ we define the $\gam$-inflated $n$-fold cross-conformal prediction set based on the training data $\Tn$ and a new feature vector $x_{n+1}$ as
    \begin{align}
        \Pcc{\alpha}{\gam} = \left\{y \in \mathcal{Y}_n:
            1 + \sum_{i=1}^{n} \mathds{1}\{\mathcal{C}_n((y, x_{n+1}), \Tnlt{n+1}{i}) \leq \cmt{i}{n+1} + \gam \} > \alpha(n+1)\right\}.
    \end{align}
\end{definition}

\begin{proposition}\label{prop:app_crossConfEquiSC}
    Assume $\mathfrak{C}$ is stable. Then, prediction sets based on $n$-fold cross-conformal prediction are asymptotically equivalent to the shortcut sets in the sense that \Cref{thm:sc_equivalence} holds with $\Pfc{\cdot}{\cdot}$ replaced by $\Pcc{\cdot}{\cdot}$.
\end{proposition}

Combining \Cref{prop:app_crossConfEquiSC} with \Cref{thm:sc_equivalence} immediately implies the following result.
\begin{corollary}[Equivalence of cross-conformal and full conformal under stability]\label{cor:app_CrossConfEquiFC}
    Assume $\mathfrak{C}$ is stable. Then, prediction sets based on $n$-fold cross-conformal prediction are asymptotically equivalent to full conformal prediction sets in the sense that \Cref{thm:sc_equivalence} holds with $\Pfcs{\cdot}{\cdot}$ replaced by $\Pcc{\cdot}{\cdot}$.
    
    If the conformity score is bounded in probability, then $\gam$-inflated prediction sets based on cross-conformal prediction are asymptotically uniformly conditionally conservative for $\gam > 0$ in the sense of \Cref{thm:ma_asymptoticResult}. Furthermore, under the additional Assumption CC, the non-inflated cross-conformal prediction sets are asymptotically uniformly conditionally valid in the sense of \Cref{thm:ma_asymptoticValidity}.
\end{corollary}

\Cref{cor:app_CrossConfEquiFC} shows that $n$-fold cross-conformal prediction comes with the same asymptotic coverage guarantees as full conformal prediction when based on stable conformity scores and is therefore another computationally advantageous alternative to full conformal prediction.
However, the major advantage of our shortcut formula is that it needs no more model refits than $n$-fold cross-conformal prediction in general and can be approximated much more efficiently for unimodal conformity scores (cf. \Cref{prop:app_algorithm}).

\subsection{Asymptotic equivalence under stability}
We summarize our findings combined with existing results for the Jackknife versions in the following statement.
\begin{corollary}\label{cor:summarizeEquivalences}
    Let $PS^{1}(\cdot) \leftrightharpoons PS^{2}(\cdot)$ denote the asymptotic equivalence between two methods of creating prediction sets $PS^{1}(\cdot)$ and $PS^{2}(\cdot)$ in the sense of \Cref{thm:sc_equivalence}.
    \begin{enumerate}[label=\alph*)]
        \item\label{it:FcScCc} If $\mathfrak{C}$ is stable, we have
        \begin{align*}
            \Pfc{\cdot}{\cdot} \leftrightharpoons \Pfcs{\cdot}{\cdot} \leftrightharpoons \Pcc{\cdot}{\cdot},
        \end{align*}
        where $\Pcc{\cdot}{\cdot}$ denotes the $n$-fold cross-conformal method.
        
        \item\label{it:Jackknifes} (\cite{amann2025UQ}:) For every symmetric prediction algorithm $\mathcal{A}$ that is asymptotically out-of-sample stable as in \Cref{def:app_AlgStability}, we have
        \begin{align*}
            PI^{sJ}_{\mathcal{A}}(\cdot) \leftrightharpoons PI^{sJ+}_{\mathcal{A}}(\cdot),
        \end{align*}
        where $PI^{sJ}_{\mathcal{A}}(\cdot)$ and $PI^{sJ+}_{\mathcal{A}}(\cdot)$ denote prediction sets based on the symmetric Jackknife and the symmetric Jackknife+, respectively, with prediction algorithm $\mathcal{A}$.
        
        \item\label{it:allMethodsEquivalent} Assume the conformity score $\mathcal{C}_n^{out}$ is defined as the (absolute) out-sample prediction error of a symmetric algorithm $\mathcal{A}$.
            \begin{enumerate}[label=\roman*)]
                \item\label{it:allME_JackSc} Then, the two prediction intervals $PI^{sJ}_{\mathcal{A}}(\cdot)$ and $\Pfcs{\cdot}{\cdot}$ coincide deterministically.
                \item\label{it:allME_all} If $\mathcal{A}$ is asymptotically out-of-sample stable, then all four methods (full conformal, $n$-fold cross-conformal, Jackknife, Jackknife+) are asymptotically equivalent, that is,
                    \begin{align*}
                        \Pfc{\cdot}{\cdot} \leftrightharpoons \Pcc{\cdot}{\cdot} \leftrightharpoons \Pfcs{\cdot}{\cdot} = PI^{sJ}_{\mathcal{A}}(\cdot) \leftrightharpoons PI^{sJ+}_{\mathcal{A}}(\cdot).
                    \end{align*}
                \item\label{it:allME_conservative} If, additionally, the prediction error is bounded in probability, then the inflated versions of all four methods are asymptotically uniformly conditionally conservative in general, and the non-inflated versions are asymptotically uniformly conditionally valid in the continuous case.\footnote{
                    In particular, this shows that the four methods are equivalent to the oracle prediction set $PI_{\alpha}^{oracle}(\gam) = (\mathcal{A}(x_{n+1}, \Tn) \pm \Q{1-\alpha}{H}) + \gam)$, where $H(t) = \PC{|y_{n+1} - \mathcal{A}(x_{n+1}, \Tn)| \leq t}{\Tn}$.
                }
            \end{enumerate}
    \end{enumerate}
\end{corollary}

One drawback of \Cref{cor:jackknifeEquiv} and \Cref{cor:summarizeEquivalences}~\ref{it:allMethodsEquivalent} is the restriction to the conformity score $\mathcal{C}_n^{out}$. Indeed, there is also a (weaker) connection between the Jackknife and the full conformal method based on $\mathcal{C}_n^{in}$ (cf. \Cref{thm:app_fcBecomesJackknife}).

\subsection{Comparison with the literature}\label{sub:dis_comparison}
A similar version of \Cref{thm:ma_asymptoticResult} has appeared in the literature.
In fact, \textcite{liang2025algorithmic} showed that $\gam$-inflated full conformal prediction sets (with $\gam > 0$) are asymptotically conditionally conservative for \emph{fixed} $\alpha \in \R$, if the conformity score $\mathcal{C}_n^{in}$ is used (cf. Theorem~$3.7$ therein).
\Cref{thm:ma_asymptoticResult} is a generalization in several ways:
Firstly, the special case of $\mathcal{C}_n^{in}$ is contained in our more general framework, allowing for a huge class of conformity scores.\footnote{
    Note that our stability condition requires $|y_{n+1} - \mathcal{A}(x_{n+1}, \Tnp)| - |y_{n+1} - \mathcal{A}(x_{n+1}, \Tnl{n})|$ to converge to $0$ in this special case. A sufficient condition for this is $|\mathcal{A}(x_{n+1}, \Tnp) - \mathcal{A}(x_{n+1}, \Tnl{n})| \plim 0$. Thus, our stability condition is comparable to, but slightly weaker than, the one in \textcite{liang2025algorithmic}.
}
Secondly, we can prove a stronger coverage guarantee which holds uniformly over all possible $\alpha$ rather than pointwise, which allows to choose $\alpha$ depending on the training data at the minor cost of additionally assuming a stochastically bounded conformity score.
Thirdly, we can show that the non-inflated prediction sets are asymptotically conditionally valid in the continuous case, while \textcite{liang2025algorithmic} focused on the case of $\gam$-inflated prediction sets only (with $\gam > 0$).

Furthermore, \textcite{ndiaye2022stable} proposed a shortcut formula for full conformal prediction sets that is slightly different from ours.
However, to provide valid coverage guarantees for the shortcut formulas provided therein, the conformity scores must fulfill a worst-case stability condition, which implies that out-of-sample prediction errors can be accurately approximated by in-sample prediction errors. Under this condition, even the naive prediction interval based on fitted errors is conditionally conservative in large samples (cf. Proposition~B.$2$ in \cite{amann2025UQ}). However, this worst-case stability condition typically fails to hold in the high-dimensional regime where the dimension of the feature vector $p$ is not negligible compared to the number of observations $n$, which is a common setting in data science.

\section{Summary}
The present paper shows that inflated full conformal prediction sets are asymptotically conditionally conservative if the sequence of conformity scores is stable and bounded in probability. Our results are based on finite-sample PAC bounds that make no assumption about the distribution $P_n$ of the feature-response pair.
In the asymptotic framework, we
allow $P_n$ to depend on $n$ and therefore include the challenging high-dimensional framework as a special case.
Furthermore, we showed that under a continuity assumption, the non-inflated full conformal prediction sets are even asymptotically conditionally valid, i.e., their training conditional coverage probability converges to the nominal level in expectation rather than just not undershooting it.

Furthermore, we propose a shortcut formula for full conformal prediction sets that provides the same asymptotic conditional coverage guarantees for stable conformity scores. We demonstrate the computational advantages of the shortcut formula for the conformity score defined as the absolute in-sample prediction error for several cases of interest. If the conformity score is defined as the absolute out-of-sample prediction error, it coincides with the symmetrized Jackknife, thereby providing a connection between the Jackknife and full conformal prediction. Additionally, we showed that for general stable conformity scores, $n$-fold cross-conformal prediction sets satisfy the same asymptotic conditional coverage guarantees as full conformal prediction.
Thus, if the conformity score is defined as the absolute out-of-sample prediction error based on a stable prediction algorithm, full conformal prediction, $n$-fold cross-conformal prediction, the Jackknife, and the Jackknife+ satisfy the same conditional coverage probabilities in large samples.

%% file: appendix.tex
\section{Further results}\label{sec:furtherResults}
\subsection{Finite sample results}
Let $\Tnls{n+1}{1}$ be defined as $(\tilde{t}_1, t_2, \ldots, t_n)$ where $\tilde{t}_1$ is an i.i.d. copy of $t_1$ that is independent of $\Tnp$.
\begin{theorem}[Finite sample results for conformal prediction]\label{thm:ma_finiteSampleResult}
    For every $\gam, \epsilon_1, \epsilon_2 > 0$ we have
    \begin{align*}
        \P\left( \sup_{\alpha \in [0,1]} \PC{y_{n+1} \notin \PI{\gam}}{\Tn} - \alpha \geq \epsilon_1 + \epsilon_2 \right)
        &\leq \dfrac{3\E|\cm{n+1} - \cms{n+1}{1}|}{\gam \epsilon_1^2 \epsilon_2}  \\
        &+ \dfrac{\E|\cm{1}|}{(n+1)\gam \epsilon_1^2 \epsilon_2}.  
    \end{align*}
    Furthermore, we also have for every $K > 0$
    \begin{align*}
        &\P\left( \sup_{\alpha \in [0,1]} \PC{y_{n+1} \notin \PI{\gam}}{\Tn} - \alpha \geq \epsilon_1 + \epsilon_2 \right)
        \leq \dfrac{1}{\gam \epsilon_1 \epsilon_2}\P(|\cm{n+1}| \geq K)\\
        &+\dfrac{3}{\gam \epsilon_1^2 \epsilon_2} \E( \min(2K + 3\gam, |\cm{n+1} - \cms{n+1}{1}|))
        + \dfrac{2K + 3 \gam}{2(n+1)\gam \epsilon_1^2 \epsilon_2}.
    \end{align*}    
\end{theorem}

\begin{remark}[Finite sample results for non-inflated full-conformal sets]
    \Cref{thm:ma_finiteSampleResult} is a statement for $\gam > 0$, that is, for inflated prediction sets. To draw the connection to non-inflated full conformal prediction sets, we assume that the continuous case assumption of \Cref{def:set_continuity} is fulfilled. We then have for every $\alpha \in [0,1]:$
    \begin{align*}
        &\PC{y_{n+1} \in \PI{\gam}}{\Tn}
        = \PC{\cm{n+1} \leq \Q{1-\alpha}{\hat{F}} + \gam}{\Tn}\\
        &\leq \PC{\cm{n+1} \leq \Q{1-\alpha}{\hat{F}}}{\Tn}
        + \min(1, \gam \supnorm{f_{\Tn}})\\
        &= \PC{y_{n+1} \in \PI{0}}{\Tn} + \min(1, \gam \supnorm{f_{\Tn}}) \text{ a.s.}
    \end{align*}
    Combining this inequality with \Cref{thm:ma_finiteSampleResult} gives finite sample results for the non-inflated full conformal prediction sets.
\end{remark}

\subsection{On the equivalence result of \Cref{thm:sc_equivalence}}
\begin{proposition}\label{prop:dis_setDifference}
    Assume $\mathfrak{C}$ is stable.
    Furthermore, let $f_0 > 0$ such that for all $n \in \N$ the conditional distribution of $y^{(n)}_{1}$ given $x^{(n)}_1 = x$ is absolutely continuous for almost all $x \in \mathcal{X}_n$ with a density $f^{(n)}_{y|x}$ that is bounded from below by $f_0$ over $\mathcal{Y}_n$.\footnote{Note that this implies that the length of $\mathcal{Y}_n \subset \R$ is bounded from above by $1/f_0$.}
    \begin{itemize}
        \item General case: For all $\alpha \in (0,1)$, $\gam_1 < \gam_2$ and $\epsilon \in (0, \min(\alpha, 1-\alpha)]$ we have
            \begin{align*}
                \limN \E[\lambda(\Pfc{\alpha}{\gam_1} \backslash \Pfcs{\alpha-\epsilon}{\gam_2})] = 0 \text{ and } 
                \limN \E[\lambda(\Pfc{\alpha}{\gam_2} \backslash \Pfcs{\alpha+\epsilon}{\gam_1})] = 0.
            \end{align*}
        \item Continuous case: If, additionally, the CC assumption is fulfilled and $(\mathcal{C}_n(t_{n+1}, \Tn))_{n \in \N}$ is stochastically bounded, then
            the length of the symmetric difference between the full conformal prediction set and the shortcut formula converges to $0$ as $n \to \infty$, that is,
            \begin{align*}
                \limN \E[\lambda(\PI{0} \Delta \PIs{0})] = 0 \text{ for all } \alpha \in [0,1].
            \end{align*}
    \end{itemize}
\end{proposition}

\subsection{Out-of-sample vs. in-sample}
We start with the following lemma.
\begin{lemma}\label{lem:dis_InToOut}
    Let $P_n^x$ be nonatomic. Then, for every symmetric algorithm $\mathcal{A}$, there exists another symmetric algorithm $\mathcal{B}$ whose out-of-sample prediction and its leave-one-out predictions coincide with those of the original algorithm almost surely, i.e.,
    $\mathcal{B}(x_{n+1}, \Tn) = \mathcal{A}(x_{n+1}, \Tn)$ and $\mathcal{B}(x_i, \Tnlt{n+1}{i}) = \mathcal{A}(x_i, \Tnlt{n+1}{i})$ a.s. for all $1 \leq i \leq n$.
    Furthermore, the in-sample predictions of $\mathcal{B}$ coincide with the leave-one-out predictions of $\mathcal{A}$ a.s., that is,
    $\mathcal{B}(x_i, \Tn) = \mathcal{A}(x_i, \Tnlt{n+1}{i})$.
\end{lemma}

As \Cref{lem:dis_InToOut} shows, there is a close connection between the in-sample and out-of-sample prediction errors given that the distribution of $x_{n+1}$ is nonatomic. This leads to the following result:
\begin{theorem}\label{thm:app_fcBecomesJackknife}
    Let $P_n^x$ be nonatomic, $\mathcal{A}$ be a symmetric prediction algorithm and
    denote the associated prediction interval based on the symmetrized Jackknife with $PI^J_{\mathcal{A}}$.
    Consider the algorithm $\mathcal{B}$ that is explicitly constructed from $\mathcal{A}$ as in the proof of \Cref{lem:dis_InToOut}, and denote the associated full conformal prediction set using the in-sample prediction error $\mathcal{C}_n^{in}$ based on the prediction algorithm $\mathcal{B}$ as the conformity score with $PS^{fc}_{\mathcal{B}}$. 
    
    If $\mathcal{A}$ is asymptotically out-of-sample stable,
    then $PI^J_{\mathcal{A}}$ coincides with the shortcut formula for $PS^{fc}_{\mathcal{B}}$ and $\mathfrak{C}$ is stable.
    Thus, $PS^{fc}_{\mathcal{B}}$ is asymptotically equivalent to $PI^J_{\mathcal{A}}$ in the sense of \Cref{thm:sc_equivalence}.
\end{theorem}

Loosely speaking, \Cref{thm:app_fcBecomesJackknife} can be interpreted as follows: Assume the distribution of $x_{n+1}$ is nonatomic and $\mathcal{A}$ is asymptotically out-of-sample stable.
Then, for the given distribution, the Jackknife method cannot substantially outperform full conformal prediction in terms of conditional coverage in large samples because there is always a full conformal prediction set with similar conditional coverage probabilities as its Jackknife analogue.
Furthermore, this conformal prediction set can be explicitly constructed from $\mathcal{A}$.
In this sense, the full conformal prediction method can be viewed as a \emph{richer} class of conditionally conservative/valid prediction sets.

As the following result shows, the condition that requires the distribution of $x_{n+1}$ to be nonatomic can be easily met in practice by adding a randomly generated unique identifier to each feature vector.
\begin{remark}\label{rem:dis_uniqueIdentifyer}
    The assumption on the distribution of $x_{n+1}$ to be nonatomic can be met very easily in practice: For $1 \leq i \leq n$, let $u_i$ be i.i.d. random variables of a uniform distribution on $[0,1]$ that are independent of the training data. Then, defining the new vector $\tilde{x}_i = (x_i^T, u_i)^T$ gives a unique identifier to each feature, whose distribution is now nonatomic.
\end{remark}

\Cref{lem:dis_InToOut} provides a method to create an in-sample stable prediction algorithm that is based on an out-of-sample stable prediction algorithm.
For the sake of completeness, we here provide an approach for the reverse direction, that is, creating an out-of-sample stable prediction algorithm that is based on an in-sample stable prediction algorithm.
\begin{lemma}\label{lem:dis_OutToIn}
    Let the distribution of $x_{n+1}$ be nonatomic. Then, for every symmetric algorithm $\mathcal{A}$ there is another symmetric algorithm $\tilde{\mathcal{A}}$, such that $\tilde{\mathcal{A}}(x_{i}, \Tn) = \mathcal{A}(x_i, \Tn)$ almost surely for all $1 \leq i \leq n$ and 
    \begin{align*}
        \E \left| \tilde{\mathcal{A}}(x_{n+2}, \Tnp) - \tilde{\mathcal{A}}(x_{n+2}, \Tn) \right| &\leq \E \left| \mathcal{A}(x_{1}, \Tnp) - \mathcal{A}(x_{1}, \Tn) \right|\\
        &+ \dfrac{1}{n+1} \E \left| \mathcal{A}(x_{1}, \Tnp) - \mathcal{A}(x_{2}, \Tnp) \right|.
    \end{align*}
\end{lemma}

\section{Algorithms}\label{sec:algorithm}
We summarize the strength of \Cref{alg:main} in the following proposition:

\begin{proposition}\label{prop:app_algorithm}
    Fix training data $\Tn$, a regressor $x_{n+1}$, a unimodal conformity score $\mathcal{C}_n^{in}$,
    $\alpha \in (0,1)$ and $\gam \in \R$.
    Then, \Cref{alg:main} applied with boundary parameter $K \in \Z$ and tolerance level $\epsilon \in (0, 2^K]$ creates a (possibly unbounded) interval $PI$, such that the following holds true:
    \begin{itemize}
        \item $\PIs{\gam} \subseteq PI$
        \item If $\emptyset \neq \PIs{\gam} \subseteq [-2^K + \epsilon, 2^K - \epsilon]$, then the approximation set $PI$ is longer than $\PIs{\gam}$ by no more than $2\epsilon$, that is, $\lambda(PI \backslash \PIs{\gam}) \leq 2 \epsilon$, where $\lambda$ denotes the Lebesgue measure.
        \item Let $\phi$ denote the golden ratio $\frac{\sqrt{5}+1}{2}$ and $\lfloor c \rfloor$ the largest integer not exceeding $c > 0$. Then, the prediction algorithm needs to be refitted no more than
        \begin{align*}
            \left\lfloor 10 + \left(K+1 + \log_2\left(\frac{1}{\epsilon}\right)\right)\left(2+\dfrac{1}{\log_2(\phi)}\right) \right\rfloor
            \leq 13.45 + 3.45 \log_2 \left( \dfrac{2^K}{\epsilon} \right)
        \end{align*}
        times.
    \end{itemize}
\end{proposition}

In principle, \Cref{alg:main} can be used for \emph{any} unimodal conformity score. The advantage of using $\mathcal{C}_n^{in}$ only arises in the calculation of $\hat{G}$, which is the empirical distribution function of the fitted errors. Thus, for other unimodal conformity scores \Cref{alg:main} potentially needs to be refitted additional $n-1$ times\footnote{
    Note that in the case of \Cref{prop:app_algorithm}, also one fit is needed to calculate $\cmt{i}{n+1}$ for $1 \leq i \leq n$. Thus, in the general case, it takes (up to) $n-1$ more fits.
} 
for the calculation of $\cmt{i}{n+1}$.

\begin{algorithm}
    \caption{Bisection algorithm}\label{alg:bisection}
    \begin{algorithmic}
        \Require Tolerance level $\epsilon > 0$, threshold $b$ and $L,U \in \R$ such that
        $\mathcal{C}((U, x_{n+1}), \Tn)) > b \geq \mathcal{C}((L, x_{n+1}), \Tn))$
        \Ensure Find $l,u \in \R$ such that $|l - u| \leq \epsilon$ and
        $\mathcal{C}((u, x_{n+1}), \Tn)) > b \geq \mathcal{C}((l, x_{n+1}), \Tn))$.
        \Procedure{BiSection}{$L,U$}
        \State $l \gets L$, $u \gets U$
        \While{$|l-u| > \epsilon$}
            \State $m \gets \frac{l+u}{2}$
            \If{$\mathcal{C}((m, x_{n+1}), \Tn)) > b$}
                \State $u \gets m$
            \Else
                \State $l \gets m$
            \EndIf
        \EndWhile
        \State \textbf{return} $\{l,u\}$.
        \EndProcedure
    \end{algorithmic}
\end{algorithm}

\begin{algorithm}
    \caption{Approximate minimizer of a unimodal conformity score}\label{alg:minimizer}
    \begin{algorithmic}
        \Require Tolerance level $\epsilon > 0$ and $L < U \in \R$.
        \Ensure Find $m, M \in \R$ such that the minimizer of the conformity score on the interval $[L,U]$ 
            is contained in $[\min(m, M), \max(m, M)]$, $|m - M| \leq \epsilon$ and the conformity score evaluated at $m$ is no larger than $M$
        \Procedure{Minimizer}{$L,U$}
            \State $c_L \gets \mathcal{C}((L, x_{n+1}), \Tn), c_U \gets \mathcal{C}((U, x_{n+1}), \Tn)$, $g \gets \dfrac{\sqrt{5}-1}{2}$
            \State $s \gets L + (1-g)(U-L)$, $t \gets L + g(U-L)$
            \State $c_s \gets \mathcal{C}((s, x_{n+1}), \Tn), c_t \gets \mathcal{C}((t, x_{n+1}), \Tn)$
            \While{$U-L > \epsilon$}
                \If{$c_s > c_t$}
                    \State $L \gets s$, $c_L \gets c_s$, $s \gets t$
                    \State $t \gets L+g(U-L)$, $c_s \gets c_t$, $c_t \gets \mathcal{C}((t, x_{n+1}), \Tn)$
                \Else
                    \State $U \gets t$, $c_U \gets c_t$, $t \gets s$
                    \State $s \gets L+(1-g)(U-L)$, $c_t \gets c_s$, $c_s \gets \mathcal{C}((s, x_{n+1}), \Tn)$
                \EndIf
            \EndWhile
            \If{$c_L < c_U$}
                \State \textbf{return} $\{L, U\}$
            \Else
                \State \textbf{return} $\{U, L\}$
            \EndIf
        \EndProcedure
    \end{algorithmic}
\end{algorithm}

\begin{algorithm}
    \caption{Shortcut algorithm for unimodal conformity scores}\label{alg:main}
    \begin{algorithmic}
        \Require (Nominal) miscoverage probability $\alpha \in (0,1)$, inflation parameter $\gam \in \R$, 
        training data $\Tn$, regressor $x_{n+1}$, unimodal conformity score $\mathcal{C}_n$, tolerance level $\epsilon > 0$, boundary parameter $K \in \N$.
        \Ensure Find $L, U \in \R$, such that $\PIs{\gam} \subseteq [L,U]$. Additionally, $[L + \epsilon, U - \epsilon] \subseteq \PIs{\gam}$ whenever $\PIs{\gam} \subseteq [-2^K + \epsilon, 2^K - \epsilon]$.
        \State $b \gets \Q{1-\alpha}{\hat{G}} + \gam$
        \State $c_1 \gets \mathcal{C}((-2^K, x_{n+1}), \Tn)$, $c_2 \gets \mathcal{C}((2^K, x_{n+1}), \Tn)$

        \If{$\max(c_1, c_2) \leq b$}
            \State \textbf{return} $(-\infty, \infty)$
        \ElsIf{$c_1 \leq b < c_2$}
            \State $\{l,u\} \gets$ \textsc{BiSection}$(-2^K, 2^{K}, b)$
            \State \textbf{return} $(-\infty, u]$
        \ElsIf{$c_2 \leq b < c_1$}
            \State $\{l,u\} \gets$ \textsc{BiSection}$(2^K, -2^{K}, b)$
            \State \textbf{return} $[u, \infty)$
        \Else \Comment{$b < \min(c_1, c_2)$}
            \State $\{m, M\} \gets $ \textsc{Minimizer}$(-2^{K}, 2^{K})$ 
                \Comment{Approximate minimizer on $[-2^{K}, 2^{K}]$ up to $\pm \epsilon$.}
            \If{$m = -2^K$} \Comment{Global minimum is left of $-2^K+\epsilon$.}
                \State \textbf{return} $(-\infty, -2^{K}+\epsilon)$
            \ElsIf{$m = 2^K$} \Comment{Global minimum is right to $2^K-\epsilon$.}
                \State \textbf{return} $(2^{K}-\epsilon, \infty)$
            \Else \Comment{Global minimum is attained within $[-2^K, 2^K]$.}
                \If{$\mathcal{C}((m, x_{n+1}), \Tn) > b$}
                    \State \textbf{return} $(\min(m, M), \max(m, M))$ 
                \Else \Comment{Found two outer points $-2^K, 2^K$ and an inner point $m$.}
                    \State $\{l_0,l_1\} \gets $\textsc{BiSection}$(m,-2^K)$
                    \State $\{u_0,u_1\} \gets $\textsc{BiSection}$(m,2^K)$
                    \State \textbf{return} $[l_1, u_1]$
                \EndIf
            \EndIf
        \EndIf
    \end{algorithmic}
\end{algorithm}

\clearpage

\section{The \ldnamEnd}\label{sec:ld}
Our proofs heavily rely on the \ldnamEnd, a technique which was introduced in \textcite{amann2025UQ}. For convenience, we will repeat the definition and provide some properties we will use in this section. 

\begin{definition}[\ldnamEnd]\label{def:ld_levyDivergence}
    Let $F: \R \to [0,1]$ and $G: \R \to [0,1]$ be cumulative distribution functions and $\gam \geq 0$. 
    We then define the \ldname between $F$ and $G$ with tolerance parameter $\gam$ as
    \begin{align}\label{eq:ld_levyDivergenceDefinition}
        \LD = \inf\{\epsilon \geq 0: F(t- \gam) - \epsilon \leq G(t) \leq F(t+\gam) + \epsilon \fatir\}.
    \end{align}
\end{definition}

For comparison, we present the definition of the L\'{e}vy metric  $L(F,G)$ between $F$ and $G$:
\begin{align*}
    L(F,G) = \inf\{\epsilon \geq 0: F(t- \epsilon) - \epsilon \leq G(t) \leq F(t+\epsilon) + \epsilon \fatir\}.
\end{align*}

The \ldname uses an externally chosen tolerance level $\gam \geq 0$ while the amount of shifting of $t$ is only internally given in the definition of the L\'{e}vy metric. 
In the next lemma, we list some properties of the \ldnamEnd:
\begin{lemma}[Lemma~C.2 of \textcite{amann2025UQ}]\label{lem:ld_basicProperties}
    Let $F$ and $G$ be cumulative distribution functions and $\gam \geq 0$. 
    Then the following results hold:
    \leavevmode
    \begin{lemmenum} 
        \item \label{it:ld_bpInfimumAttained}
            The infimum in \cref{eq:ld_levyDivergenceDefinition} is attained. In particular, this yields
            \begin{align}\label{eq:ld_ldAttainsMinimum}
                F(t- \gam) - \LD \leq G(t) \leq F(t+\gam) + \LD \fatir.
            \end{align}
        \item \label{it:ld_bpSymmetry}
            The \ldname is symmetric $\ld{\gam}(F,G) = \ld{\gam}(G,F)$. 
        \item \label{it:ld_bpMonoCont}
            The function $\gam \mapsto \LD$ is nonincreasing and continuous from the right on $[0, \infty)$. 
            In particular, we have $0 \leq \LD \leq \ld{0}(F,G) = \supnorm{F-G} \leq 1$.
        \item \label{it:ld_bpAltDef}
            The \ldname can be equivalently written as 
            \begin{align}\label{eq:ld_ldAlternativeDefinition}
                \LD = \sup_{t \in \R} \max( F(t) - G(t+\gam), G(t) - F(t+\gam)).
            \end{align}
        \item \label{it:ld_bpTriangle}
            Let $H$ be another distribution function and $\epsilon \geq 0$. 
            We then have the following inequality:
            \begin{align}\label{eq:ld_ldTriangleInequality}
                \ld{\gam + \epsilon}(F,H) \leq \ld{\gam}(F,G) + \ld{\epsilon}(G,H).
            \end{align}
        \item \label{it:ld_bpLevy}
            The \ldname is connected to the L\'{e}vy metric as follows:
            \begin{align}\label{eq:ld_connectionToLevy}
                \min(\gam, \LD) \leq L(F,G) \leq \max(\gam, \LD) \text{ for every } \gam \geq 0.
            \end{align}
        \item \label{it:ld_bpScaling}
            Let $c > 0$ and let $F(c \cdot)$ and $G(c \cdot)$ 
            denote the functions $t \mapsto F(c t)$ and $t \mapsto G(c t)$, respectively.
            Then the \ldname fulfills the following scaling invariance:
            \begin{align}\label{eq:ld_ScalingInvariance}
                \LD =\ld{\frac{\gam}{c}}(F(c \cdot), G(c \cdot)),
            \end{align}
    \end{lemmenum}
\end{lemma}

Changing our perspective from distribution functions to their corresponding quantiles, we get the following result, which yields the key argument in our proofs:
\begin{proposition}[Proposition~C.4 in \textcite{amann2025UQ}]\label{prop:ld_quantileInequality}
    Let $F: \R \to [0,1]$ and $G: \R \to [0,1]$ be cumulative distribution functions and $\gam \geq 0$. 
    We then have:
    \begin{align}\label{eq:ld_quantileInequality}
        \Q{\alpha-\LD}{F} - \gam \leq \Q{\alpha}{G} \leq \Q{\alpha + \LD}{F} + \gam \text{ for all } \alpha \in \R
    \end{align}
    with the convention that $-\infty \leq c -\infty$ and $+\infty \geq c +  \infty$ for all $c \in \R$.
\end{proposition}

\begin{lemma}[Lemma~C.10 of \textcite{amann2025UQ}]\label{lem:appLd_boundingLdSquared}
    Let $F$ and $G$ be cumulative distribution functions, $\gam > 0$, $K \geq 0$ and $\mu \in \R$. We then have the following bounds for the \ldname between $F$ and $G$:
    \begin{align*}
        &\LD \leq 1 - F(K+\mu) + F(-K+\mu) + \left[ \dfrac{1}{\gam} \int_{[-K+\mu-\gam, K+\mu+2\gam]} |F(x) - G(x)|^2 d\lambda(x) \right]^{\frac{1}{2}},\\
        &\LD^2 \leq \dfrac{1}{\gam} \int_{\R} |F(x) - G(x)|^2 d\lambda(x).
    \end{align*}
\end{lemma}

\section{Proofs}\label{sec:proofs}
\subsection{Proofs for \Cref{sec:fullConformal} and \Cref{thm:ma_finiteSampleResult}}\label{sub:appProofsForFC}
\begin{proof}[Proof of \Cref{lem:set_marginalCoverage}]
    It is enough to prove the case $\gam = 0$ since we have $\PI{0} \subseteq \PI{\gam}$ for all $\gam \geq 0$.
    The case $\gam = 0$ can be found in \textcite{vovk2022algorithmic} (cf. Proposition~$2.3$ therein). For the sake of completeness, we present a proof that has been adapted to our notation.
    
    We abbreviate $\hat{F}_{y_{n+1}}$ with $\hat{F}$, that is,
    \begin{align*}
        \hat{F} = \dfrac{1}{n+1}\sum_{i=1}^{n+1}\mathds{1}\{\cm{i} \leq t\}.
    \end{align*}
    By definition, $y_{n+1} \in \PI{0}$ if, and only if, $\cm{n+1} \leq \Q{1-\alpha}{\hat{F}}$.
    Define $l_i = \cm{i}$ for $1 \leq i \leq n+1$ and let $l_{(1)} \leq \ldots, l_{(n+1)}$ denote the ordered conformal values. Thus, we can rewrite $\cm{n+1} \leq \Q{1-\alpha}{\hat{F}}$ as $l_{n+1} \leq l_{(\lceil (n+1)(1-\alpha) \rceil)}$.
    Since the feature-response pairs $t_i$ are i.i.d. and the conformity score is symmetric in its second argument, 
    the random variables $l_1, \ldots, l_{n+1}$ are exchangeable. 
    Thus, we have
    \begin{align*}
        \P\left(l_{n+1} \leq l_{(\lceil (n+1)(1-\alpha) \rceil)}\right) \geq \dfrac{\lceil (n+1)(1-\alpha) \rceil}{n+1} \geq 1-\alpha.
    \end{align*}
\end{proof}

We need some auxiliary results before we prove \Cref{thm:ma_asymptoticResult}. 
\begin{lemma}\label{lem:appLd_expIntIndFunctions}[Lemma~D.1 of \textcite{amann2025UQ}]
    Let $A \in \mathcal{B}(R)$ be a Borel set and $X$ and $Y$ be random variables.
    We then have
    \begin{align*}
        \int_A \E_{X,Y} \left( |\mathds{1}\{X \leq t\} - \mathds{1}\{Y \leq t\}| \right) d\lambda(t)
        \leq \E_{X,Y}\left( \min(\lambda(A), |X-Y| )\right),
    \end{align*}
    where $\lambda$ denotes the Lebesgue measure.
\end{lemma}

The following result is the core of \Cref{thm:ma_asymptoticResult} as it provides a pointwise bound for the quantities of interest.
\begin{proposition}\label{lem:app_pointwiseBound}
    Let $M > 0$ and $f: \mathcal{Z}\times\mathcal{Z}^n \to [M_1, M_2]$ be a measurable bounded function 
    which is symmetric in its second argument.
    Denoting $f(t_i, \Tnl{i})$ with $k_i$, we define
    \begin{align}
    	F &= \EC{k_{n+1}}{\Tn} \\
    	\hat{F} &= \dfrac{1}{n+1} \sum_{i=1}^{n+1}k_i.
    \end{align}
    Let $k_{n+1,1}$ denote $f(t_{n+1}, (\tilde{t}_1, t_2, \ldots, t_n))$ 
    where $\tilde{t}_1$ is an i.i.d copy of $t_1$ which is independent of $\Tnp$.
    Then the following upper bound holds true:
    \begin{nalign}\label{eq:app_pointwiseBoundInequality}
    	\E \left( (F - \hat{F})^2 \right)
    	&\leq \dfrac{3n^2 + 4n + 2}{(n+1)^2} (M_2 - M_1) \E \left| k_{n+1} - k_{n+1,1} \right| \\
    	&+ \dfrac{(M_2 - M_1)}{2(n+1)} \E(|k_1 - k_2|).
    \end{nalign}
\end{proposition}

\begin{proof}
W.l.o.g. we may assume $M_1=0$ as otherwise we can shift the function $f$ by $M_1$ and $M$ denote $M_2 - M_1 = M_2$.\footnote{With a similar argument we could also additionally assume that $M = 1$.}
For $1 \leq i \neq j \leq n+1$ we define 
$k_{i,j}$ similar to $k_i$ with the exception that $t_j$ is replaced by an i.i.d. copy $\tilde{t}_j$ of $t_j$ which is independent of all training data $\Tnp$, i.e., $k_{i,j} = f(t_i, \Tnls{i}{j})$ with $\Tnls{i}{j}$ denoting the training data $\Tnl{i}$ where $t_j$ has been replaced by $\tilde{t}_j$.
Thus, $k_{i,j}$ has the same distribution as $k_i$, but is independent of $t_j$.
Furthermore, let denote
\begin{align}
	\Delta = \E(| k_{n+1} - k_{n+1,1}|).
\end{align}
Since the data are i.i.d. we immediately conclude $\Delta = \E(|k_{i} - k_{i,j}|)$ for each $1 \leq i \neq j \leq n+1$. 
We start with the observation that
\begin{align*}
	\E \left( (F - \hat{F})^2 \right)
	&= \E( F^2 ) 
	- \dfrac{2}{n+1} \sum_{i=1}^{n+1} \E( Fk_i ) 
	+ \dfrac{1}{(n+1)^2} \sum_{i=1}^{n+1} \sum_{j=1}^{n+1} \E( k_i k_j)\\
	&= \dfrac{1}{n+1} \sum_{i=1}^{n+1} \E( F^2 - Fk_i) 
	+ \dfrac{1}{(n+1)^2} \sum_{i,j: i \neq j} \E(k_i (k_j - F) \\
	&+ \dfrac{1}{(n+1)^2} \sum_{i=1}^{n+1} E(k_i^2 - k_i F).
\end{align*}

We will deal with the three terms in the last equality separately. 
Starting with the first one, we stress the fact that $F$ is independent of $t_{n+1}$ and we therefore have $\E(k_{n+1}F) = \E( \EC{k_{n+1}}{\Tn} F) = \E(F^2)$.
Now, let $i < n+1$. We then have
\begin{align*}
	\E(k_{i,n+1}k_{n+1,i})
    &= \E \left(\EC{k_{i,n+1}}{\Tnls{i}{n+1}} 
    - \EC{k_{n+1,i}}{\Tnls{n+1}{i}} \right).
\end{align*}
We would like to emphasize that both $\EC{k_{i,n+1}}{\Tnls{i}{n+1}}$ and $\EC{k_{n+1,i}}{\Tnls{n+1}{i}}$ are independent of $t_i$ and $t_{n+1}$. Thus, we have
\begin{align*}
    \E(F^2) 
    &= \E(Fk_{n+1})
    = \E( k_{n+1} \EC{k_{n+1}}{\Tn})\\
    &= \E( k_{n+1} \EC{k_{n+1,i}}{\Tnls{n+1}{i}}
    + \E( k_{n+1} 
    \left( \EC{k_{n+1}}{\Tn})
    - \EC{k_{n+1,i}}{\Tnls{n+1}{i}} \right).
\end{align*}
We claim that the last summand can be bounded from above by $M \Delta$. To see this, we proceed as follows:
\begin{nalign}\label{eq:app_PointwiseBoundEq1}
    &\E( k_{n+1} 
    \left( \EC{k_{n+1}}{\Tn})
    - \EC{k_{n+1,i}}{\Tnls{n+1}{i}} \right) )\\
    &\leq \E( M \left| \EC{k_{n+1}}{\Tn})
    - \EC{k_{n+1,i}}{\Tnls{n+1}{i}} \right| )
    \leq M \E( |k_{n+1} - k_{n+1,i}|) = M \Delta.
\end{nalign}
Since $\EC{k_{n+1,i}}{\Tnls{n+1}{i}}$ is independent of $t_{n+1}$ and $t_i$
we conclude
\begin{align*}
   \E( k_{n+1} \EC{k_{n+1,i}}{\Tnls{n+1}{i}})
   &= \E( k_i \EC{k_{n+1,i}}{\Tnls{n+1}{i}})\\
   &= \E(k_{i,n+1} \EC{k_{n+1,i}}{\Tnls{n+1}{i}}),
\end{align*}
where we exchanged $t_i$ with $t_{n_1}$ in the first equality and replaced $t_{n+1}$ by an i.i.d. copy $\tilde{t}_{n+1}$ in the second equality.
With the same argument as in \cref{eq:app_PointwiseBoundEq1}
we have
\begin{align*}
    \left| \E(k_{i,n+1} \EC{k_{n+1,i}}{\Tnls{n+1}{i}})
    - \E(k_{i,n+1} \EC{k_{n+1}}{\Tn}) \right|
    \leq M \Delta.
\end{align*}
Putting the pieces together, we end up with
\begin{align}\label{eq:app_PointwiseBoundEq2}
    \E(F^2) 
    \leq 2 M \Delta + \E(k_{i,n+1} \EC{k_{n+1}}{\Tn})
    = 2M \Delta + \E(k_i F),
\end{align}
where we used the fact that $F$ is independent of $t_{n+1}$ in the last equality.

We now turn to $\E(k_i (k_j - F)$ for $i \neq j$.
For this, we again use the argumentation as in \cref{eq:app_PointwiseBoundEq1} to conclude that
\begin{align*}
    \left| \E(k_{i,j}k_j) - \E(k_i k_j) \right| 
    \leq M \Delta.
\end{align*}
Furthermore, we observe that 
\begin{align*}
    \E(k_{i,j}k_j) 
    = \E(k_{i,j} \EC{k_j}{\Tnl{j}}
    = \E(k_i \EC{k_j}{\Tnl{j}}).
\end{align*}
Since $k_i$ is symmetric in its second argument, we can exchange $t_j$ with $t_{n+1}$ to get
\begin{align*}
    \E(k_i \EC{k_j}{\Tnl{j}})
    = \E(k_i \EC{k_{n+1}}{\Tn})
    = \E(k_i F).
\end{align*}
To sum it up, we have shown that
\begin{align}\label{eq:app_PointwiseBoundEq3}
    \E(k_i k_j) 
    \leq M \Delta + \E(k_{i,j}k_j) 
    = M \Delta + \E(k_i F).
\end{align}

Next, we turn to the third term and start with the case $i = n+1$.
For this, we note that for any $j < n+1$ we have
\begin{align*}
    \E(k_{n+1} F)
    &= \E(k_{n+1} \EC{k_{n+1}}{\Tn}))
    \geq \E(k_{n+1} \EC{k_{n+1,j}}{\Tnls{n+1}{j}}) - M \Delta\\
    &= \E(k_j \EC{k_{n+1,j}}{\Tnls{n+1}{j}}) - M \Delta
    \geq \E(k_j \EC{k_{n+1}}{\Tn}) - 2M \Delta\\
    &= \E(k_j F) - 2M \Delta,
\end{align*}
where we used an inequality as in \cref{eq:app_PointwiseBoundEq1} twice.
Thus, we conclude
\begin{align*}
    \E(k_{n+1}^2 - k_{n+1}F)
    \leq \E(k_{n+1}^2) - \E(k_j F) + 2M \Delta
    = \E(k_j^2 - k_j F) + 2M \Delta.
\end{align*}
In other words, we have linked the case $i = n+1$ to the cases where $i < n+1$ (by setting $j = i$).
Next, we take any $j < n+1$ to get
\begin{align*}
    \E(k_i^2 - k_i F)
    &\leq \E(k_i \max(0, k_i - F))
    \leq M \E( \max(k_i - F))\\
    &= M \E( \max(k_{i,n+1} - \EC{k_{n+1}}{\Tn}, 0)).
\end{align*}
By the convexity of the function $g(x) = \max(0, a - x)$ and Jensen's inequality we conclude that
\begin{align*}
    M \E\left( \max(k_{i,n+1} - \EC{k_{n+1}}{\Tn}, 0)\right)
    &\leq M \E\left( \max(k_{i,n+1} - k_{n+1}), 0 \right)\\
    &\leq M \E\left( \max(k_i - k_{n+1}, 0) + |k_{i,n+1} - k_i| \right)\\
    &\leq M \Delta + M \E( \max(k_i - k_{n+1}, 0).
\end{align*}
    However, since the data are i.i.d. the term $k_i - k_{n+1}$ has the same distribution as $k_{j} - k_l$ for every $j \neq l$ we conclude
\begin{align*}
    M \E( \max(k_i - k_{n+1}, 0)
    &= \dfrac{M}{2} \left( \E( \max(k_i - k_{n+1}, 0) + \E( \max(k_{n+1} - k_{n+i}, 0) \right)\\
    &= \dfrac{M}{2} \E( |k_i - k_{n+1}| )
    = \dfrac{M}{2} \E( |k_1 - k_2|).
\end{align*}
To sum it up, we have shown that
\begin{align*}
    \E(k_i^2 - k_i F) \leq M \Delta + \dfrac{M}{2} \E(|k_1 - k_2|).
\end{align*}
Combining the inequality of the preceding display with \cref{eq:app_PointwiseBoundEq2} and \cref{eq:app_PointwiseBoundEq3} we conclude
\begin{align*}
    \E( (F - \hat{F})^2) 
    &\leq \dfrac{2n}{n+1}M \Delta + \dfrac{n}{n+1}M \Delta + \dfrac{2M\Delta}{(n+1)^2} + \dfrac{1}{n+1}\left( M \Delta + \dfrac{M}{2} \E(|k_1 - k_2|) \right)\\
    &= M \Delta \dfrac{3n^2 + 4n + 3}{(n+1)^2} + \dfrac{M}{2(n+1)} \E(|k_1 - k_2|),
\end{align*}
which finishes the proof.
\end{proof}

By definition, the \ldname provides a uniform inequality (cf. \ref{it:ld_bpInfimumAttained}). The following lemma provides a way to turn the specific pointwise bound derived in \Cref{lem:app_pointwiseBound} into a uniform one for the \ldnamEnd:
\begin{lemma}\label{lem:ld_genericUpperBound}
    Let $F(t) = \P_{t_{n+1}}(Z_F \leq t\|\Tn)$ and $G(t) = \P_{t_{n+1}}(Z_G \leq t\|\Tn)$ be cdfs 
    conditional on the training data $\Tn$ 
    where $Z_F$ and $Z_G$ are random variables being measurable w.r.t. to $\Tnp$
    and assume we have the following pointwise bound for all $t \in \R$:
    \begin{align*}
        \E((F(t)-G(t))^2) \leq 
        C_1 \E\left| \mathds{1}\{X_1 \leq t\} - \mathds{1}\{Y_1 \leq t\}\right|
        + C_2 \E\left|\mathds{1}\{X_2 \leq t\} - \mathds{1}\{Y_2 \leq t\}\right|,
    \end{align*}
    where $C_1, C_2$ are positive constants (independent of $F,G$ and $t$), 
    and $X_i$ and $Y_i$ are random variables for $i \in \{1,2\}$.
    We then have for any $\gam > 0$, $\epsilon > 0$:
    \begin{align*}
        \P( \LD \geq \epsilon) \leq 
        \dfrac{C_1}{\gam \epsilon^2} \E|X_1 - Y_1| + \dfrac{C_2}{\gam \epsilon^2} \E|X_2 - Y_2|.
    \end{align*}
    Furthermore, we also have for any $K > 0$:
    \begin{align*}
        \P( \LD \geq 2\epsilon) 
        &\leq \dfrac{1}{\gam \epsilon}\P(|Z_F| \geq K)
        +\dfrac{C_1}{\gam \epsilon^2} \E( \min(2K + 3\gam, |X_1 - Y_1|)) \\
        &+ \dfrac{C_2}{\gam \epsilon^2} \E( \min(2K + 3\gam, |X_2 - Y_2|)).
    \end{align*}    
\end{lemma}

\begin{proof}[Proof of \Cref{lem:ld_genericUpperBound}]
    Applying \Cref{lem:appLd_boundingLdSquared} yields
    $$\LD^2 \leq \dfrac{1}{\gam} \int_{\R} (F(x) - G(x))^2 d\lambda(x)$$
    and therefore Markov's inequality, together with Tonelli's Theorem, gives
    \begin{align*}
        \P( \LD \geq \epsilon) \leq \dfrac{\E(\LD^2)}{\epsilon^2}
        \leq \dfrac{1}{\gam \epsilon^2} \int_R \E(F(x) - G(x))^2 d\lambda(x).
    \end{align*}
    By assumption, we can bound this from above by
    \begin{align*}
        \dfrac{1}{\gam \epsilon^2} \int_R C_1 \E\left| \mathds{1}\{X_1 \leq t\} - \mathds{1}\{Y_1 \leq t\}\right|
        + C_2 \E\left|\mathds{1}\{X_2 \leq t\} - \mathds{1}\{Y_2 \leq t\}\right| d\lambda(x).
    \end{align*}
    Applying \Cref{lem:appLd_expIntIndFunctions} with $A = \R$ this gives the the following upper bound:
    \begin{align*}
        \dfrac{C_1}{\gam \epsilon^2} \E|X_1 - Y_1| + \dfrac{C_2}{\gam \epsilon^2} \E|X_2 - Y_2|,
    \end{align*}
    which proves the first part.
    
    For the second part, we use the second inequality of \Cref{lem:appLd_boundingLdSquared} to get
    \begin{align*}
        \LD \leq 1 - F(K) + F(-K) 
    + \left[ \dfrac{1}{\gam} \int_{[-K-\gam, K+2\gam]} |F(x) - G(x)|^2 d\lambda(x) \right]^{\frac{1}{2}}
    \end{align*}
    Now, we have
    \begin{align*}
        &\P( \LD \geq 2\epsilon) \\
        &\leq  \P( 1 - F(K) + F(-K) \geq \epsilon) 
        + \P\left( \left[ \dfrac{1}{\gam} \int_{[-K-\gam, K+2\gam]} |F(x) - G(x)|^2 d\lambda(x) \right]^{\frac{1}{2}} \geq \epsilon \right)\\
        &\leq \dfrac{\E(1 - F(K) + F(-K))}{\epsilon} 
        + \dfrac{1}{\gam \epsilon^2} \E\left( \int_{[-K-\gam, K+2\gam]} |F(x) - G(x)|^2 d\lambda(x) \right),
    \end{align*}
    where the last line can be derived from Markov's inequality.
    For the first part, we point out that 
    $$\E(1 - F(K) + F(-K)) = \E_{\Tn}(1 - \P_{t_{n+1}}(-K < Z_F \leq K\|\Tn) = 1 - \P(-K < Z_F \leq K)$$
    holds. For the second part, we again apply Tonelli's Theorem together with \Cref{lem:appLd_expIntIndFunctions} but this time with $A = [-K-\gam, K+2\gam]$, which yields
    \begin{align*}
        \E\left( \int_{[-K-\gam, K+2\gam]} |F(x) - G(x)|^2 d\lambda(x) \right)
        &\leq \dfrac{C_1}{\gam \epsilon^2} \E(\min(2K + 3\gam, |X_1 - Y_1|)) \\
        &+ \dfrac{C_2}{\gam \epsilon^2} \E( \min(2K + 3\gam, |X_2 - Y_2|))
    \end{align*}
\end{proof}

\begin{proof}[Proof of \Cref{thm:ma_finiteSampleResult}]
    Let $\epsilon_1 \geq 0$ and $\epsilon_2 > 0$ and fix training data $\Tnp$. 
    We define two distribution functions $F: \R \to [0,1]$ and $\hat{F}: \R \to [0,1]$ as follows:
    \begin{nalign}\label{eq:ma_finiteSampleResultDefineFFhat}
        F(t) &= \PC{\cm{n+1} \leq t}{\Tn}\\
        \hat{F}(t) &= \dfrac{1}{n+1}\sum_{i=1}^{n+1}\mathds{1}\{\cm{i} \leq t\}.
    \end{nalign}
    Recalling the definition given by \Cref{eq:set_defineFullConfSets}, we have
    \begin{align*}
        y_{n+1} \in \PI{\gam} \Longleftrightarrow \cm{n+1} \leq \Q{1-\alpha}{\hat{F}} + \gam.
    \end{align*}
    In view of \Cref{prop:ld_quantileInequality} we conclude for any $\alpha \in [0,1]$
    \begin{align*}
        \cm{n+1} \leq \Q{1 - \alpha - \ld{\gam}(F, \hat{F})}{F}
        \Longrightarrow y_{n+1} \in \PI{\gam}.
    \end{align*}
    Thus, on the event $\ld{\gam}(F, \hat{F}) \leq \epsilon_1$ we have
    \begin{align*}
        \cm{n+1} \leq \Q{1 - \alpha - \epsilon_1}{F}
        \Longrightarrow y_{n+1} \in \PI{\gam}.
    \end{align*}
    To sum it up, we have shown that for every fixed $\Tnp$ we have
    \begin{align*}
        \mathds{1}\{y_{n+1} \in \PI{\gam}\} 
        &\geq \mathds{1}\{\cm{n+1} \leq \Q{1 - \alpha - \epsilon_1}{F}\} \\
        &- \mathds{1}\{\ld{\gam}(F, \hat{F}) > \epsilon_1\}.
    \end{align*}
    Taking the expectation with respect to $(y_{n+1}, x_{n+1})$ we have
    \begin{align*}
        \PC{y_{n+1} \in \PI{\gam}}{\Tn} 
        &\geq \PC{\cm{n+1} \leq \Q{1 - \alpha - \epsilon_1}{F}}{\Tn} \\
        &- \PC{\ld{\gam}(F, \hat{F}) > \epsilon_1}{\Tn}.
    \end{align*}
    Recalling the definition of $F$ and the definition of quantiles given by \Cref{def:ld_quantiles} we conclude
    \begin{align*}
        \PC{\cm{n+1} \leq \Q{1 - \alpha - \epsilon_1}{F}}{\Tn} 
        = F(\Q{1 - \alpha - \epsilon_1}{F})
        \geq 1 - \alpha - \epsilon_1
    \end{align*}
    which also holds true if $1 - \alpha - \epsilon_1 \leq 0$.
    Since $\alpha \in [0,1]$ was arbitrary this even yields
    \begin{align*}
        \sup_{\alpha \in [0,1]} 1 - \alpha - \PC{y_{n+1} \in \PI{\gam}}{\Tn}
        \leq \epsilon_1 + \PC{\ld{\gam}(F, \hat{F}) > \epsilon_1}{\Tn}.
    \end{align*}
    
    Thus, we ultimately have
    \begin{nalign}
        &\P\left( \sup_{\alpha \in [0,1]} \PC{y_{n+1} \notin \PI{\gam}}{\Tn} - \alpha \geq \epsilon_1 + \epsilon_2 \right)\\
        &\leq \P( \PC{\ld{\gam}(F, \hat{F}) > \epsilon_1}{\Tn} \geq \epsilon_2)
        \leq \dfrac{\P\left(\ld{\gam}(F, \hat{F}) > \epsilon_1\right)}{\epsilon_2},
    \end{nalign}
    where we used Markov's inequality in the last step.

    In order to find an upper bound for $\P(\ld{\gam}(F, \hat{F}) > \epsilon_1)$ we will combine \Cref{lem:app_pointwiseBound} with \Cref{lem:ld_genericUpperBound} as follows:
    Fix $s \in \R$ and define $f_s: \mathcal{Z}\times{\mathcal{Z}^n} \to [0,1]$ as follows:
    $f_s(t_i, \Tnl{i}) = \mathds{1}\{\cm{i} \leq s\}$.
    Then, \Cref{lem:app_pointwiseBound} yields
    \begin{align*}
        \E((F(s) - \hat{F}(s))^2)
        &\leq 3 \E | \mathds{1}\{\cm{n+1} \leq s\} - \mathds{1}\{\cms{n+1}{1} \leq s\}\\
        &+ \dfrac{1}{2(n+1)} \E | \mathds{1}\{\cm{1} \leq s\} - \mathds{1}\{\cm{2} \leq s\}|,
    \end{align*}
    where we used the fact that $3n^2 + 4n + 2 \leq 3(n+1)^2$ to simplify the inequality.
    Now, applying the first inequality of \Cref{lem:ld_genericUpperBound} we end up with
    \begin{nalign}\label{eq:app_finiteSampleResultEq1}
        \P(\ld{\gam}(F, \hat{F}) > \epsilon_1) 
        &\leq \dfrac{3\E|\cm{n+1} - \cms{n+1}{1}|}{\gam \epsilon_1^2}  \\
        &+ \dfrac{\E|\cm{1}-\cm{2}|}{2(n+1)\gam \epsilon_1^2}.
    \end{nalign}

    To finish the proof of the first part of \Cref{thm:ma_finiteSampleResult}, we use the fact that 
    $$\E|\cm{1}-\cm{2}| \leq 2\E|\cm{1}|.$$

    For the second part of \Cref{thm:ma_finiteSampleResult} we use the second inequality of \Cref{lem:ld_genericUpperBound} to get
    \begin{nalign}\label{eq:app_finiteSampleResultEq2}
        \P(\ld{\gam}(F, \hat{F}) > \epsilon_1)
        \leq& \dfrac{1}{\gam \epsilon_1}\P(|\cm{n+1}| \geq K)\\
        &+\dfrac{3}{\gam \epsilon_1^2} \E( \min(2K + 3\gam, |\cm{n+1} - \cms{n+1}{1}|)) \\
        &+ \dfrac{1}{2(n+1)\gam \epsilon_1^2} \E( \min(2K + 3\gam, |\cm{1}-\cm{2}|)).
    \end{nalign}
    To finish the proof we stress the fact that the last summand can be bounded from above by $\dfrac{2K+3\gam}{2(n+1)\gam \epsilon_1^2}$.
\end{proof}

\begin{proof}[Proof of \Cref{thm:ma_asymptoticResult}]
    For every fixed $n \in \N$ we apply \Cref{thm:ma_finiteSampleResult} with $\epsilon_1 = \epsilon_2 = \epsilon/2$.
    In the first case, the upper bound of \Cref{thm:ma_finiteSampleResult} converges to $0$ by assumption. 
    For the second case, we start with an arbitrarily small $\kappa > 0$. Since the conformity score is bounded in probability, we can find a $K_\kappa > 0$, such that \begin{align*}
        \sup_{n \in \N} \dfrac{\P(|\cm{n+1}| \geq K_n)}{\gam \epsilon^3} \leq \kappa.
    \end{align*}
    Furthermore, we have $\limN \frac{2K_\kappa + 3 \gam}{2(n+1)\gam \epsilon^3} = 0$. 
    
    The triangle inequality yields
    \begin{align*}
        &\E( \min(2K_\kappa + 3\gam, |\cm{n+1} - \cms{n+1}{1}|))\\
        &\leq 
        \E( \min(2K_\kappa + 3\gam, |\cm{n+1} - \cmt{n+1}{1}|)) \\
        &+ \E( \min(2K_\kappa + 3\gam, |\cmt{n+1}{1} - \cms{n+1}{1}|)).
    \end{align*}
    Note that $\cm{n+1} - \cmt{n+1}{1}$ and $\cms{n+1}{1} - \cmt{n+1}{1}$ have the same distribution as the instability coefficient $\cm{n+1} - \cmt{n+1}{n}$.
    Since the latter converges to $0$ in probability, we conclude
    \begin{align*}
        \limN \dfrac{3}{\gam \epsilon^3} \E( \min(2K_\kappa + 3\gam, |\cm{n+1} - \cms{n+1}{1}|)) = 0.
    \end{align*}
    To sum it up, we have shown that for every $\kappa > 0$, we have
    \begin{align*}
        \limN \P\left( \sup_{\alpha \in [0,1]} \PC{y_{n+1} \notin \PI{\gam}}{\Tn} - \alpha \geq \epsilon \right) \leq \kappa.
    \end{align*}
    Since $\kappa > 0$ can be made arbitrarily small, the statement follows.
\end{proof}

\begin{proof}[Proof of \Cref{thm:ma_asymptoticValidity}]
    Recalling the definition given by \Cref{eq:set_defineFullConfSets} we have
    \begin{align*}
        y_{n+1} \in \PI{0} \Longleftrightarrow \cm{n+1} \leq \Q{1-\alpha}{\hat{F}},
    \end{align*}
    where $\hat{F}$ and $F$ are defined as in \Cref{eq:ma_finiteSampleResultDefineFFhat}.
    By \Cref{eq:ld_quantileInequality} we have for any $\gam > 0$
    \begin{align}\label{eq:app_asymptoticValidityEq1}
        \Q{1-\alpha-\ld{\gam}(F, \hat{F})}{F} - \gam \leq  \Q{1-\alpha}{\hat{F}}
        \leq \Q{1-\alpha +\ld{\gam}(F, \hat{F})}{F} + \gam,
    \end{align}
    which yields for any $\epsilon \in (0, \alpha)$
    \begin{align*}
        \PC{y_{n+1}\in \PI{0}}{\Tn} 
        &\leq \PC{\cm{n+1} \leq \Q{1-\alpha + \epsilon}{F} + \gam}{\Tn} \\
        &+ \PC{\ld{\gam}(F, \hat{F}) > \epsilon}{\Tn}.
    \end{align*}
    From the continuity case assumption CC, we conclude
    \begin{align*}
        &\PC{\cm{n+1} \leq \Q{1-\alpha + \epsilon}{F} + \gam}{\Tn}\\
        &\leq \PC{\cm{n+1} \leq \Q{1-\alpha + \epsilon}{F}}{\Tn}
        + \min(1, \supnorm{f_{\Tn}}\gam)\\
        &= F(\Q{1-\alpha + \epsilon}{F}) + \min(1, \supnorm{f_{\Tn}}\gam).
    \end{align*}
    Since $F$ is absolutely continuous for almost every $\Tn$, we have $F(\Q{1-\alpha + \epsilon}{F}) = 1 - \alpha + \epsilon$ almost surely.
    Putting the pieces together, we end up with the inequality
    \begin{align*}
        \PC{y_{n+1}\in \PI{0}}{\Tn} \leq 
        \PC{\ld{\gam}(F, \hat{F}) > \epsilon}{\Tn} + 1 - \alpha + \epsilon + \min(1, \supnorm{f_{\Tn}}\gam)
    \end{align*}
    almost surely.
    Since the inequality in the preceding display trivially holds for every $\epsilon \geq \alpha$, we can weaken the restriction $\epsilon \in (0, \alpha)$ to $\epsilon > 0$.
    We emphasize that the inequality of the preceding display holds for every $n \in \N$. 
    Combining \cref{eq:app_finiteSampleResultEq1} or \cref{eq:app_finiteSampleResultEq2} with the assumptions of \Cref{thm:ma_asymptoticResult} we conclude that for any fixed $\gam > 0$ and $\epsilon > 0$ the expression $\P(\ld{\gam}(F, \hat{F}) > \epsilon)$ converges to $0$ for $n \to \infty$.   
    Thus, we can find null-sequences $(\gam_n)_{n \in \N}$ and $(\epsilon_n)_{n \in \N}$, such that 
    $\limN \P(\ld{\gam_n}(F, \hat{F}) > \epsilon_n) = 0$, which implies
    $\PC{\ld{\gam_n}(F, \hat{F}) > \epsilon_n}{\Tn} \plim 0$.
    By the boundedness in probability of $\supnorm{f_{\Tn}}$ this also yields
    \begin{align*}
        \min(1, \supnorm{f_{\Tn}}\gam_n) \plim 0.
    \end{align*}
    To sum it up, we have shown that for every $\epsilon > 0$ we have
    \begin{align*}
        \limN \P( \PC{y_{n+1}\in \PI{0}}{\Tn} \geq 1 - \alpha + \epsilon) = 0.
    \end{align*}
    The same argument can be used for the lower bound of \Cref{eq:app_asymptoticValidityEq1} to show that also
    \begin{align*}
        \limN \P( \PC{y_{n+1}\in \PI{0}}{\Tn} \leq 1 - \alpha - \epsilon) = 0,
    \end{align*}    
    which finishes the proof.\footnote{Alternatively, we could have argued that the marginal coverage probability of full conformal prediction sets is bounded from below by $1-\alpha$, which, together with the first conditional bound, also provides a lower bound for the conditional coverage probability.} 
\end{proof}

\begin{remark}
Although \Cref{thm:ma_asymptoticValidity} is similar in its assumptions and results to Theorem $4.5$ and Corollary $5.6$ in \textcite{amann2025UQ} and Theorem $3.6$ in \textcite{barber2021predictive}, which provide an equivalent statement for the Jackknife and the Jackknife+, its proof comes with a crucial change:
For the Jackknife and its variants, we know that the $\gam$-inflated prediction intervals are larger than their non-inflated prediction intervals by the (asymptotically arbitrarily small) amount of $2\gam$. However, this need not be true for the $\gam$-inflated prediction set based on the full conformal method. In fact, the length of the $\gam$-inflated prediction set can be arbitrarily larger than the length of the original, non-inflated prediction set, which is the reason why the proof of \Cref{thm:ma_asymptoticValidity} is more involved than it might be anticipated.
\end{remark}

\subsection{Proofs for \Cref{sec:shortcut}}
We start with an auxiliary lemma.
\begin{lemma}\label{lem:app_scFgDifference}
    Let $\hat{F} = \hat{F}^{y_{n+1}}$ and $\hat{G}$ be defined as in \Cref{eq:ma_finiteSampleResultDefineFFhat} and \Cref{def:sc_Pis}, respectively.
    We then have for all $\gam \geq 0$:
    \begin{align*}
        \ld{\gam}(\hat{F}, \hat{G}) \leq \dfrac{1}{n+1} + \dfrac{1}{n+1} \sum_{i=1}^n \mathds{1}\{|\cm{i} - \cmt{i}{n+1}| > \gam\}.
    \end{align*}
\end{lemma}

\begin{proof}
    Define $\kappa = \dfrac{1}{n+1} + \dfrac{1}{n+1} \sum_{i=1}^n \mathds{1}\{|\cm{i} - \cmt{i}{n+1}| > \gam\}$. We start showing that $\hat{F}(t) \leq \hat{G}(t + \gam) + \kappa$ for all $t \in \R$:
    \begin{align*}
        \hat{F}(t) &= \dfrac{1}{n+1} \sum_{i=1}^{n+1} \mathds{1}\{\cm{i} \leq t\}
        \leq \dfrac{1}{n+1} + \dfrac{1}{n+1} \sum_{i=1}^n \mathds{1}\{\cm{i} \leq t\} \\
        &\leq \dfrac{1}{n+1} + \dfrac{1}{n+1}\sum_{i=1}^n \mathds{1}\{\cmt{i}{n+1} \leq t + \gam\} + \mathds{1}\{\cmt{i}{n+1} > \cm{i} + \gam\}\\
        &\leq \dfrac{n}{n+1}\hat{G}(t+\gam) + \kappa.
    \end{align*}
    Next, we show that $\hat{G}(t) \leq \hat{F}(t + \gam) + \kappa$ holds for all $t \in \R$:
    \begin{align*}
        \hat{G}(t) &= \dfrac{1}{n} \sum_{i=1}^n \mathds{1}\{\cmt{n+1}{i} \leq t\} \\
        &\leq \dfrac{1}{n} \sum_{i=1}^n \mathds{1}\{\cm{i} \leq t + \gam\}
        + \mathds{1}\{\cm{i} > \cmt{i}{n+1} + \gam\}\\
        &\leq \dfrac{n+1}{n} \hat{F}(t+\gam) + \dfrac{1}{n} \sum_{i=1}^n \mathds{1}\{\cm{i} > \cmt{i}{n+1} + \gam\}
        \leq \hat{F}(t+\gam) + \kappa.
    \end{align*}
    Now the claim follows from \Cref{eq:ld_ldAlternativeDefinition}.
\end{proof}

\begin{proof}[Proof of \Cref{thm:sc_equivalence}]
    We start with preliminary considerations.
    For fixed $\Tnp$, $\alpha \in \R$ and $\gam \geq 0$, \Cref{def:sc_Pis} yields
    \begin{align*}
        y_{n+1} \in \PIs{2\gam} \Longleftrightarrow \cm{n+1} \leq \Q{1-\alpha}{\hat{G}} + 2\gam.
    \end{align*}
    By \Cref{prop:ld_quantileInequality}, we have for every $\gam \geq 0$
    \begin{align*}
        \Q{1-\alpha - \ld{\gam}(\hat{F}, \hat{G})}{\hat{F}} - \gam \leq \Q{1-\alpha}{\hat{G}} \leq \Q{1-\alpha + \ld{\gam}(\hat{F}, \hat{G})}{\hat{F}} + \gam
    \end{align*}where $\hat{F} = \hat{F}^{y_{n+1}}$ is defined as in \cref{eq:ma_finiteSampleResultDefineFFhat}.
    In other words, we have
    \begin{align*}
        &\left[ \cm{n+1} \leq  \Q{1-\alpha - \ld{\gam}(\hat{F}, \hat{G})}{\hat{F}} + \gam \right]\\
        &\Longrightarrow \left[ \cm{n+1} \leq \Q{1-\alpha}{\hat{G}} + 2 \gam \right]\\
        &\Longrightarrow \left[ \cm{n+1} \leq  \Q{1-\alpha + \ld{\gam}(\hat{F}, \hat{G})}{\hat{F}} + 3\gam \right].
    \end{align*}
    From this, we conclude for every $\epsilon \in (0,1)$ that on the event where $\ld{\gam}(\hat{F}, \hat{G}) \leq \epsilon$ we have
    \begin{align*}
        \mathds{1}\{y_{n+1} \in \Pfc{\alpha+\epsilon}{\gam}\} \leq \mathds{1}\{y_{n+1} \in \Pfcs{\alpha}{2\gam}\} \leq \mathds{1}\{y_{n+1} \in \Pfc{\alpha-\epsilon}{3\gam}\} 
    \end{align*}
    uniformly for all $\alpha \in \R$. In particular, this yields
    \begin{align*}
        &\P\left( \sup_{\alpha \in [0,1]} \PC{y_{n+1} \notin \Pfcs{\alpha}{2\gam}}{\Tn} - \alpha \geq 3\epsilon  \right)\\
        &\leq \P\left( \sup_{\alpha \in [0,1]} \PC{y_{n+1} \notin \Pfc{\alpha+\epsilon}{\gam}}{\Tn} - \alpha - \epsilon \geq \epsilon  \right)\\
        &+ \P\left( \PC{\ld{\gam}(\hat{F}, \hat{G}) > \epsilon}{\Tn} > \epsilon \right).
    \end{align*}
    For the last expression, we use Markov's inequality twice to get
    \begin{align*}
        \P\left( \PC{\ld{\gam}(\hat{F}, \hat{G}) > \epsilon}{\Tn} > \epsilon \right)
        \leq \epsilon^{-2}\E(\ld{\gam}(\hat{F}, \hat{G})).
    \end{align*}
    By \Cref{lem:app_scFgDifference}, we have
    \begin{align*}
        \E(\ld{\gam}(\hat{F}, \hat{G})) \leq \dfrac{1}{n+1} + \dfrac{1}{n+1} \sum_{i=1}^n \P(|\cm{i} - \cmt{i}{n+1}| > \gam|).
    \end{align*}
    Recalling that our data are i.i.d. and the conformity score is symmetric in its second argument, we conclude that $\cm{i} - \cmt{i}{n+1}$ has the same distribution as $\cm{n+1} - \cmt{n+1}{n}$. By assumption, the instability term converges to $0$ in probability, and therefore, we also have
    $\limN \E(\ld{\gam}(\hat{F}, \hat{G})) = 0$ for all $\gam > 0$.
    Thus, we have
    \begin{align*}
        &\limsup_{n \to \infty} \P\left( \sup_{\alpha \in [0,1]} \PC{y_{n+1} \notin \Pfcs{\alpha}{2\gam}}{\Tn} - \alpha \geq 3\epsilon  \right)\\
        &\leq \limN \P\left( \sup_{\alpha \in [0,1]} \PC{y_{n+1} \notin \Pfc{\alpha+\epsilon}{\gam}}{\Tn} - \alpha - \epsilon \geq \epsilon  \right)
    \end{align*}
    whenever $\gam > 0$.
    Note that for $\alpha > 1 - \epsilon$ the term $\PC{y_{n+1} \notin \Pfc{\alpha+\epsilon}{\gam}}{\Tn} - \alpha - \epsilon$ is always negative and therefore we have
    \begin{align*}
        &\limN \P\left( \sup_{\alpha \in [0,1]} \PC{y_{n+1} \notin \Pfc{\alpha+\epsilon}{\gam}}{\Tn} - \alpha - \epsilon \geq \epsilon  \right)\\
        &\leq \limN \P\left( \sup_{\alpha \in [0,1-\epsilon]} \PC{y_{n+1} \notin \Pfc{\alpha+\epsilon}{\gam}}{\Tn} - \alpha - \epsilon \geq \epsilon  \right)\\
        &\leq \limN \P\left( \sup_{\alpha \in [0,1]} \PC{y_{n+1} \notin \Pfc{\alpha}{\gam}}{\Tn} - \alpha \geq \epsilon  \right).
    \end{align*}
    This shows that \ref{it:sc_Gsc} holds whenever \ref{it:sc_Gfc} holds. 

    With the same arguments as before, it can be shown that 
    \begin{align*}
        &\limsup_{n \to \infty} \P\left( \sup_{\alpha \in [0,1]} \PC{y_{n+1} \notin \Pfc{\alpha}{\gam}}{\Tn} - \alpha \geq \epsilon  \right) \\
        &\leq \limsup_{n \to \infty} \P\left( \sup_{\alpha \in [0,1]} \PC{y_{n+1} \notin \Pfcs{\alpha - \epsilon}{2\gam}}{\Tn} - (\alpha - \epsilon) \geq 2\epsilon  \right).
    \end{align*}
    Note that for $\alpha < \epsilon$ the $(1-\alpha+\epsilon)$-quantile of any distribution function is $+\infty$, implying that $\PC{y_{n+1} \notin \Pfc{\alpha-\epsilon}{2\gam}}{\Tn} = 0$. Thus, we have
    \begin{align*}
        &\limsup_{n \to \infty} \P\left( \sup_{\alpha \in [0,1]} \PC{y_{n+1} \notin \Pfcs{\alpha - \epsilon}{2\gam}}{\Tn} - (\alpha - \epsilon) \geq 2\epsilon  \right)\\
        &= \limsup_{n \to \infty} \P\left( \sup_{\alpha \in [\epsilon,1]} \PC{y_{n+1} \notin \Pfcs{\alpha - \epsilon}{2\gam}}{\Tn} - (\alpha - \epsilon) \geq 2\epsilon  \right)\\
        &\leq \limsup_{n \to \infty} \P\left( \sup_{\alpha \in [0, 1]} \PC{y_{n+1} \notin \Pfcs{\alpha}{2\gam}}{\Tn} - \alpha \geq 2\epsilon  \right).
    \end{align*} 
    This shows that \ref{it:sc_Gfc} holds whenever \ref{it:sc_Gsc} holds. 

    Next, we consider the continuous case. 
    Let $\gam > 0$. We  have
    \begin{nalign}\label{eq:app_DeltaContinuity}
        &\PC{y_{n+1} \in \PIs{\gam}}{\Tn} - \PC{y_{n+1} \in \PIs{0}}{\Tn} \\
        &\leq \PC{\cm{n+1} \in (\Q{1-\alpha}{\hat{G}}, \Q{1-\alpha}{\hat{G}} + \gam]}{\Tn} \\
        &\leq \min(1, \gam \| f_{\Tn} \|_{\infty}) \text{ a.s.}
    \end{nalign}
    and the same holds true if we replace $\PIs{\cdot}$ by $\PI{\cdot}$.

    Thus, we have
    \begin{align*}
        &\PC{y_{n+1} \in \PIs{0}}{\Tn}  
        \geq\PC{y_{n+1} \in \PIs{\gam}}{\Tn} - \min(1, \gam \| f_{\Tn} \|_{\infty}) \\
        &\geq \PC{y_{n+1} \in \Pfc{\alpha+\epsilon}{0}}{\Tn} - \min(1, \gam \| f_{\Tn} \|_{\infty}) - \PC{\ld{\gam}(\hat{F}, \hat{G}) > \epsilon}{\Tn} \text{a.s.}
    \end{align*}

    Recall that for all $\gam > 0$ the expression $\ld{\gam}(\hat{F}, \hat{G})$ converges to $0$ in $\mathcal{L}_1$ and therefore there exists a null-sequence $\gam_n \searrow 0$ such that $\ld{\gam_n}(\hat{F}, \hat{G})$ converges to $0$ in $\mathcal{L}_1$.
    Since by assumption $\| f_{\Tn} \|_{\infty}$ is bounded in probability, we conclude that also $\min(1, \gam_n \| f_{\Tn} \|_{\infty}) \plim 0$ holds.
    Thus, we have
    \begin{align*}
        &\limN \P \left[ \sup_{\alpha \in [0,1]} \PC{y_{n+1} \notin \PIs{0}}{\Tn}  - \alpha \geq 3 \epsilon \right]\\ 
        &\leq \limN \P \left[ \sup_{\alpha \in [0,1]} \PC{y_{n+1} \notin \Pfc{\alpha+\epsilon}{0}}{\Tn}  - \alpha \geq 2\epsilon \right]\\
        &+ \limN \P\left[ \min(1, \gam_n \| f_{\Tn} \|_{\infty}) + \PC{\ld{\gam_n}(\hat{F}, \hat{G}) > \epsilon}{\Tn} \geq \epsilon \right].
    \end{align*}
    Again, we use the fact that for $\alpha + \epsilon > 1$ the expression $\PC{y_{n+1} \notin \Pfc{\alpha+\epsilon}{0}}{\Tn}  - \alpha - \epsilon$ is negative and therefore we have
    \begin{align*}
        &\limN \P \left[ \sup_{\alpha \in [0,1]} \PC{y_{n+1} \notin \Pfc{\alpha+\epsilon}{0}}{\Tn}  - (\alpha+\epsilon) \geq 2\epsilon \right]\\
       &= \limN \P \left[ \sup_{\alpha \in [0,1-\epsilon]} \PC{y_{n+1} \notin \Pfc{\alpha+\epsilon}{0}}{\Tn}  - (\alpha+\epsilon) \geq \epsilon \right]\\
        &\leq \limN \P \left[ \sup_{\alpha \in [0,1]} \PC{y_{n+1} \notin \Pfc{\alpha}{0}}{\Tn}  - \alpha \geq \epsilon \right] = 0.
    \end{align*}
    With the same argument as before, we can also show that 
    \begin{align*}
        &\PC{y_{n+1} \in \PIs{0}}{\Tn}  
        \leq\PC{y_{n+1} \in \PIs{-\gam}}{\Tn} + \min(1, \gam \| f_{\Tn} \|_{\infty}) \\
        &\leq \PC{y_{n+1} \in \Pfc{\alpha-\epsilon}{0}}{\Tn} + \min(1, \gam \| f_{\Tn} \|_{\infty}) + \PC{\ld{\gam}(\hat{F}, \hat{G}) > \epsilon}{\Tn} \text{a.s.}
    \end{align*}
    Arguing as before, we conclude that
    \begin{align*}
        \limN \P \left[ \sup_{\alpha \in [0,1]} \PC{y_{n+1} \notin \PIs{0}}{\Tn}  - \alpha \leq -3 \epsilon \right] = 0.
    \end{align*}

    For the other direction, we exchange the roles of $\PI{\cdot}$ and $\PIs{\cdot}$.  
\end{proof}

\begin{proof}[Proof of \Cref{cor:validityOfShortcut}]
    This is a direct application of \Cref{thm:ma_asymptoticResult} and \Cref{thm:sc_equivalence}.
\end{proof}

\begin{proof}[Proof of \Cref{cor:jackknifeEquiv}]
    Note that the symmetry of the conformity score $\mathcal{C}_n^{out}$ translates to the symmetry of the prediction algorithm $\mathcal{A}$. Furthermore, if the conformity score $\mathcal{C}_n^{out}$ is based on an asymptotically out-of-sample stable prediction algorithm,
    its instability coefficient converges to $0$ in probability:
    \begin{align*}
        &|y_{n+1} - \mathcal{A}_{n,n}(x_{n+1}, \Tn)| - 
        |y_{n+1} - \mathcal{A}_{n,n-1}(x_{n+1}, \Tnl{n})| \\
        &\leq |\mathcal{A}_{n,n}(x_{n+1}, \Tn) - \mathcal{A}_{n,n-1}(x_{n+1}, \Tnl{n})| \plim 0.
    \end{align*}
    Thus, the assumption of \Cref{thm:sc_equivalence} is fulfilled. 
\end{proof}

\subsection{Proofs for \Cref{sec:discussion}}
\begin{lemma}\label{lem:app_SandwichLemma}
    Fix training data $\Tn$ and a regressor $x_{n+1}$. We then have for all $\alpha \in \R$, $\epsilon > 0$, $\gam_1 \in \R$ and $\gam_2 > 0$:
    \begin{align*}
        \Pcc{\alpha+\epsilon}{\gam_1} 
        &\subseteq \Pfcs{\alpha}{\gam_1 + \gam_2} 
        \cup \Delta_n(\epsilon, \gam_2)\\
        \Pfcs{\alpha+\epsilon}{\gam_1}
        &\subseteq \Pcc{\alpha}{\gam_1 + \gam_2} \cup
        \Delta_n(\epsilon, \gam_2),
    \end{align*}
    where $\Delta_n(\epsilon, \gam)$ is defined as
    \begin{align*}
        \left\{y \in \mathcal{Y}_n:
        \sum_{i=1}^n \mathds{1}\{|\mathcal{C}_n((y, x_{n+1}), \Tn) - \mathcal{C}_n((y, x_{n+1}), \Tnlt{n+1}{i})| \geq \gam\} > n\epsilon - 1\right\}.
    \end{align*}
\end{lemma}

\begin{proof}
    Assume $\alpha < 0$. Then, $\Pcc{\alpha}{\gam_1 + \gam_2}$ and $\Pfcs{\alpha}{\gam_1 + \gam_2}$ both coincide with $\mathcal{Y}_n$ and, hence, the statement is trivially fulfilled.

    If, otherwise, $\alpha + \epsilon \geq 1$, $\Pcc{\alpha + \epsilon}{\gam_1}$ both coincide with $\emptyset$ and, hence, the statement is trivially fulfilled.
    
    It remains to show the statements for the case $\alpha \in [0,1-\epsilon)$.
    We start with the first claim:
    Let 
    $y \in \Pcc{\alpha+\epsilon}{\gam_1} \backslash \Delta_n(\epsilon, \gam)$.
    Since $y$ is contained in $\Pcc{\alpha+\epsilon}{\gam_1}$, we have
    \begin{align*}
        \sum_{i=1}^{n} \mathds{1}\{\mathcal{C}_n((y, x_{n+1}), \Tnlt{n+1}{i}) \leq \cmt{i}{n+1} + \gam_2 \} > (\alpha+\epsilon)(n+1) - 1.
    \end{align*}
    Thus, $y \notin \Delta_n(\epsilon, \gam_2)$ implies
    \begin{align*}
        \sum_{i=1}^{n} \mathds{1}\{\mathcal{C}_n((y, x_{n+1}), \Tn) < \cmt{i}{n+1} + \gam_1 + \gam_2\} + n \epsilon - 1 > (\alpha+\epsilon)(n+1) - 1.
    \end{align*}    
    In other words, we have
    \begin{align*}
        n - \sum_{i=1}^{n} \mathds{1}\{\mathcal{C}_n((y, x_{n+1}), \Tn) \geq \cmt{i}{n+1} + \gam_1 + \gam_2\} > n\alpha + \alpha + \epsilon.
    \end{align*}
    Recalling that 
    \begin{align*}
        \hat{G}(t) = \dfrac{1}{n}\sum_{i=1}^n \mathds{1}_{[\mathcal{C}_{n-1}(t_i, \Tnlt{n+1}{i}), \infty)}(t),
    \end{align*}
    this implies
    \begin{align*}
        \hat{G}(\mathcal{C}_n((y, x_{n+1}), \Tn) - \gam_1 - \gam_2)
        < (1 - \alpha) - \frac{\alpha + \epsilon}{n}
        \leq 1 - \alpha + \epsilon.
    \end{align*}
    In view of \cref{eq:ld_alphaQuantilesExceedAlpha}, this shows that 
    $\mathcal{C}_n((y, x_{n+1}), \Tn) \leq \Q{1 - \alpha}{\hat{G}} + \gam_1 + \gam_2$. By \Cref{def:sc_Pis}, this is equivalent to
    $y \in \Pfcs{\alpha}{\gam_1 + \gam_2}$.

    For the reverse direction,
    take any $y \in \Pfcs{\alpha+\epsilon}{\gam_1} \backslash \Delta_n(\epsilon, \gam_2)$.
    Thus, we have
    \begin{align*}
        \mathcal{C}_n((y, x_{n+1}), \Tn) \leq \Q{1 - \alpha - \epsilon}{\hat{G}} + \gam_1.
    \end{align*}
    In view of \cref{eq:ld_alphaQuantilesExceedAlpha}, this implies
    \begin{align*}
        \hat{G}\left(\left(\mathcal{C}_n((y, x_{n+1}), \Tn) - \gam_1\right)-\right) \leq 1 - \alpha - \epsilon,
    \end{align*}
    which can equivalently be written as
    \begin{align*}
        \sum_{i=1}^n \mathds{1}\left\{\mathcal{C}_n((y, x_{n+1}), \Tn) > \mathcal{C}_{n-1}(t_i, \Tnlt{n+1}{i}) + \gam_1\right\} \leq n(1-\alpha-\epsilon).
    \end{align*}
    Rewriting the equation of the preceding display yields
    \begin{align*}
        \sum_{i=1}^n \mathds{1}\left\{\mathcal{C}_n((y, x_{n+1}), \Tn) \leq \mathcal{C}_{n-1}(t_i, \Tnlt{n+1}{i}) + \gam_1\right\} \geq n(\alpha+\epsilon).
    \end{align*}
    Since $y \notin \Delta_n(\epsilon, \gam_2)$, this implies
    \begin{align*}
        \sum_{i=1}^n \mathds{1}\left\{\mathcal{C}_n((y, x_{n+1}), \Tnlt{i}{n+1}) \leq \mathcal{C}_{n-1}(t_i, \Tnlt{n+1}{i}) + \gam_1 + \gam_2\right\} \geq n\alpha + 1,
    \end{align*}
    which shows that $y \in \Pcc{\alpha}{\gam_1 + \gam_2}$. 
    
\end{proof}

\begin{proof}[Proof of \Cref{prop:app_crossConfEquiSC}]
    We start with the general case.
    First, note that for any $\epsilon > 0$ and $\gam > 0$ we have
    \begin{align*}
        &\P\left[
        \sup_{\alpha \in [0,1]} \PC{y_{n+1} \notin \Pfcs{\alpha+\epsilon}{\gam}}{\Tn} - \alpha \geq 2\epsilon \right]\\
        &= \P\left[
        \sup_{\alpha \in [0,1-2\epsilon]} \PC{y_{n+1} \notin \Pfcs{\alpha+\epsilon}{\gam}}{\Tn} - \alpha \geq 2\epsilon \right]\\
        &\leq \P\left[
        \sup_{\alpha \in [0,1]} \PC{y_{n+1} \notin \Pfcs{\alpha}{\gam}}{\Tn} - \alpha \geq \epsilon \right].
    \end{align*}

    \Cref{lem:app_SandwichLemma} shows that for any $\epsilon > 0$ and $\gam > 0$ we have
    \begin{align*}
        \PC{y_{n+1} \in \Pcc{\alpha}{2 \gam}}{\Tn}
        - \PC{y_{n+1} \in \Pfcs{\alpha+\epsilon}{\gam}}{\Tn} \geq
        \PC{y_{n+1} \in \Delta_n(\epsilon, \gam)}{\Tn}
    \end{align*}
    for all $\alpha \in [0,1]$.
    In particular, we have
    \begin{align*}
        &\P\left[
        \sup_{\alpha \in [0,1]} \PC{y_{n+1} \notin \Pcc{\alpha}{2 \gam}}{\Tn} - \alpha \geq 3\epsilon \right]\\
        &\leq
        \P\left[
        \sup_{\alpha \in [0,1]} \PC{y_{n+1} \notin \Pfcs{\alpha+\epsilon}{\gam}}{\Tn} - \alpha \geq 2\epsilon \right]
        + \P\left( \PC{y_{n+1} \in \Delta_n(\epsilon, \gam)}{\Tn} \geq \epsilon\right)\\
        &\leq 
        \P\left[
        \sup_{\alpha \in [0,1]} \PC{y_{n+1} \notin \Pfcs{\alpha}{\gam}}{\Tn} - \alpha\geq \epsilon \right]
        + \P\left( \PC{y_{n+1} \in \Delta_n(\epsilon, \gam)}{\Tn} \geq \epsilon\right).
    \end{align*}

    By Markov's inequality we have for any $\epsilon > \frac{1}{n}$
    \begin{align*}
        &\P\left( \PC{y_{n+1} \in \Delta_n(\epsilon, \gam)}{\Tn} \geq \epsilon\right)
        \leq \frac{1}{\epsilon} \P\left(y_{n+1} \in \Delta_n(\epsilon, \gam)\right)\\
        &= \P\left( \sum_{i=1}^n \mathds{1}\{|\mathcal{C}_n((y_{n+1}, x_{n+1}), \Tn) - \mathcal{C}_n((y_{n+1}, x_{n+1}), \Tnlt{n+1}{i})| \geq \gam\} > n\epsilon - 1\right)\\
        &\leq \frac{n}{\epsilon(n\epsilon - 1)}
        \frac{1}{n} \sum_{i=1}^n \P(|\mathcal{C}_n((y_{n+1}, x_{n+1}), \Tn) - \mathcal{C}_n((y_{n+1}, x_{n+1}), \Tnlt{n+1}{i})| \geq \gam).
    \end{align*}
    Since $\mathfrak{C}$ is stable, this implies
    \begin{align*}
        \limN \P\left( \PC{y_{n+1} \in \Delta_n(\epsilon, \gam)}{\Tn} \geq \epsilon\right) = 0
    \end{align*}
    for all $\gam > 0$.
    Thus,
    \begin{align*}
        \P\left[
        \sup_{\alpha \in [0,1]} \PC{y_{n+1} \notin \Pcc{\alpha}{2 \gam}}{\Tn} - \alpha \geq 3\epsilon \right]
    \end{align*}
    converges to $0$ as $n \to \infty$ whenever 
    \begin{align*}
        \P\left[
        \sup_{\alpha \in [0,1]} \PC{y_{n+1} \notin \Pfcs{\alpha}{\gam}}{\Tn} - \alpha\geq \epsilon \right]
    \end{align*}
    does.
    The reverse direction in the general case can be proven with the same arguments.

    Next, we consider the continuous case.
    With the same argument as in \cref{eq:app_DeltaContinuity} we conclude for any $\alpha \in [0,1]$ and $\gam > 0$:
    \begin{align*}
        \PC{y_{n+1} \in \Pfcs{\alpha}{0}}{\Tn}
        \leq \PC{y_{n+1} \in \Pfcs{\alpha}{-\gam}}{\Tn}
        + \min(1, \gam \supnorm{f_{\Tn}}) \text{ a.s.}
    \end{align*}
    
    Combining this with \Cref{lem:app_SandwichLemma}, yields for every $\alpha \in \R$ and $\epsilon > 0$:
    \begin{align*}
        \PC{y_{n+1} \in \Pfcs{\alpha}{0}}{\Tn}
        &\leq \PC{y_{n+1} \in \Pcc{\alpha-\epsilon}{0}}{\Tn}
        + \min(1, \gam \supnorm{f_{\Tn}}) \\
        &+ \PC{y_{n+1} \in \Delta_n(\epsilon, \gam)}{\Tn}
        \text{ a.s.}
    \end{align*}
    
    With similar arguments, one can also show that
    \begin{align*}
        \PC{y_{n+1} \in \Pfcs{\alpha}{0}}{\Tn}
        &\geq \PC{y_{n+1} \in \Pcc{\alpha+\epsilon}{0}}{\Tn}
        - \min(1, \gam \supnorm{f_{\Tn}}) \\
        &- \PC{y_{n+1} \in \Delta_n(\epsilon, \gam)}{\Tn}
        \text{ a.s.}
    \end{align*}
    Thus, we have
    \begin{align*}
        &\P \left( \sup_{\alpha \in [0,1]} |\PC{y_{n+1} \in \Pfcs{\alpha}{0}}{\Tn} - \alpha|  \geq 3 \epsilon \right)\\
        &\leq  \P \left( \sup_{\alpha \in [0,1]} |\PC{y_{n+1} \in \Pcc{\alpha}{0}}{\Tn} - \alpha|  \geq \epsilon \right)
        + \P( \min(1, \gam \supnorm{f_{\Tn}}) \geq \epsilon)\\
        &+ \P( \PC{y_{n+1} \in \Delta_n(\epsilon, \gam)}{\Tn} \geq \epsilon).
    \end{align*}

    Assume now that cross-conformal prediction sets are asymptotically uniformly conditionally valid.
    We then have for every $\epsilon > 0$ and $\gam > 0$:
    \begin{align*}
        \limsup_{n \to \infty} \P \left( \sup_{\alpha \in [0,1]} |\PC{y_{n+1} \in \Pfcs{\alpha}{0}}{\Tn} - \alpha|  \geq 3 \epsilon \right)
        \leq \limsup_{n \to \infty} \P( \min(1, \gam \supnorm{f_{\Tn}}) \geq \epsilon).
    \end{align*}
    Since $\supnorm{f_{\Tn}})$ is bounded in probability and $\gam > 0$ can be made arbitrarily small, we conclude
    \begin{align*}
        \limN \P \left( \sup_{\alpha \in [0,1]} |\PC{y_{n+1} \in \Pfcs{\alpha}{0}}{\Tn} - \alpha|  \geq 3 \epsilon \right) = 0.
    \end{align*}

    The reverse direction can be proven by exchanging the roles of $\Pfcs{\cdot}{\cdot}$ and $\Pcc{\cdot}{\cdot}$.
\end{proof}

\begin{proof}[Proof of \Cref{cor:app_CrossConfEquiFC}]
    The asymptotic equivalence between $\Pfc{\cdot}{\cdot}$ and $\Pcc{\cdot}{\cdot}$ is a direct combination of \Cref{prop:app_crossConfEquiSC} and \Cref{thm:sc_equivalence}.
    The second statement is a consequence of \Cref{thm:ma_asymptoticResult}, \Cref{thm:ma_asymptoticValidity} and the equivalence between $\Pfc{\cdot}{\cdot}$ and $\Pcc{\cdot}{\cdot}$.
\end{proof}

\begin{proof}[Proof of \Cref{cor:summarizeEquivalences}]
    \Cref{cor:summarizeEquivalences} \ref{it:FcScCc} summarizes \Cref{thm:sc_equivalence} and \Cref{prop:app_crossConfEquiSC}.
    \Cref{cor:summarizeEquivalences} \ref{it:Jackknifes} is given by Theorem~4.1 of \cite{amann2025UQ} and Proposition~B.1 therein.
    Next, we consider the statements in \Cref{cor:summarizeEquivalences} \ref{it:allMethodsEquivalent}: \ref{it:allME_JackSc} is given by \Cref{cor:jackknifeEquiv}.
    
    If $\mathcal{A}$ is asymptotically stable, then $\mathcal{C}^{out}$ is also stable (see \Cref{cor:jackknifeEquiv}).
    Thus, \ref{it:allME_all} combines 
    \Cref{cor:summarizeEquivalences} \ref{it:FcScCc} and \Cref{cor:summarizeEquivalences} \ref{it:Jackknifes} with  \Cref{cor:summarizeEquivalences} \ref{it:allMethodsEquivalent} \ref{it:allME_JackSc}.
    Furthermore, the conformity score is stochastically bounded if, and only if, the prediction error is.
    Thus, \ref{it:allME_conservative} follows from the equivalence results in \ref{it:allME_all} combined with \Cref{thm:ma_asymptoticResult} and \Cref{thm:ma_asymptoticValidity}.

    Note that for any $\gam > 0$ the $\gam$-inflated oracle prediction interval contains the \emph{closure} of the non-inflated oracle prediction interval and the latter has a coverage probability of $H(\Q{1-\alpha}{H}) \geq 1 - \alpha$.
    In contrast, the non-inflated prediction interval has a coverage probability of $H(\Q{1-\alpha}{H}-) \leq 1-\alpha$.
    However, under Assumption CC, the function $H$ is absolutely continuous and therefore satisfies $H(\Q{1-\alpha}{H}-) = 1-\alpha$.
\end{proof}

\subsection{Proofs for the Appendix}
The proof of \Cref{thm:ma_finiteSampleResult} is presented in \Cref{sub:appProofsForFC}.

\begin{proof}[Proof of \Cref{prop:dis_setDifference}]
    We start with a preliminary consideration.
    Let $S$ be a measurable function that maps $x_{n+1}$ to a borel set $S(x_{n+1}) \subset \R$.
    We then have
    \begin{align*}
        \lambda(S(x_{n+1}))
        &= \int_{\mathcal{Y}_n} \mathds{1}\{y \in S(x_{n+1})\} d\lambda(y)
        \leq \dfrac{1}{f_0} \int_{\mathcal{Y}_n} \mathds{1}\{y \in S(x_{n+1})\} f_{y\|x = x_{n+1}}(y) d\lambda(y) \\
        &= \dfrac{1}{f_0} \PC{y_{n+1} \in S(x_{n+1})}{x_{n+1}}.
    \end{align*}
    Thus, we have for any measurable function $T$ that maps the training data $\Tn$ and a feature vector $x_{n+1}$ to a borel set
    \begin{align*}
        \E(\lambda(T(\Tn, x_{n+1})))
        \leq \frac{1}{f_0}\P(y_{n+1} \in T(\Tn, x_{n+1})).
    \end{align*}
    Therefore, it suffices to show that $\limN \P(y_{n+1} \in T(\Tn, x_{n+1})) = 0$ for appropriately chosen sets $T$.

    We start proving the general case:
    Consider the case $\ld{\gam_2 - \gam_1}(\hat{F}, \hat{G}) \leq \epsilon$. Then, $\Q{1-\alpha}{\hat{F}} + \gam_1 \leq \Q{1-\alpha+\epsilon}{\hat{G}} + \gam_2$ and therefore $y_{n+1} \in \Pfc{\alpha}{\gam_1}$ implies $y_{n+1} \in \Pfc{\alpha-\epsilon}{\gam_2-\gam_1}$. Thus, we have
    \begin{align*}
        \P(y_{n+1} \in \Pfc{\alpha}{\gam_1} \backslash \Pfc{\alpha-\epsilon}{\gam_2})
        \leq \P(\ld{\gam_2 - \gam_1}(\hat{F}, \hat{G}) > \epsilon).
    \end{align*}
    As shown in \Cref{thm:sc_equivalence}, the latter expression converges to $0$.

    With similar arguments, one can also show that
    \begin{align*}
        \P(y_{n+1} \in \Pfc{\alpha}{\gam_2} \backslash \Pfc{\alpha+\epsilon}{\gam_1}) 
        \leq
        \P(\ld{\gam_2 - \gam_1}(\hat{F}, \hat{G}) > \epsilon).
    \end{align*}

    For the continuous case, we fix $\gam_2 > 0$ and $\epsilon > 0$.
    
    Consider the case where $\ld{\gam_2}(\hat{F}, \hat{G}) \leq \epsilon$ and $\ld{\gam_2}(\hat{F}, F) \leq \epsilon$ holds. 
    Recalling \Cref{it:ld_bpTriangle}, this also yields $\ld{2\gam_2}(\hat{G},F) \leq 2 \epsilon$.
    In this case, we have
    $\Q{1-\alpha}{\hat{F}} \leq \Q{1-\alpha+\epsilon}{F} + \gam_2$ as well as $\Q{1-\alpha}{\hat{F}} \geq \Q{1-\alpha-2\epsilon}{F} - 2\gam_2$.\footnote{Note that this even holds if $\epsilon > \alpha$ and $0 < 2\epsilon < 1-\alpha$ since we defined quantiles for all $\alpha \in \R$.} 
    Thus, we have
    \begin{align*}
        y_{n+1} \in \PI{\gam} \backslash \PIs{\gam}
        \Longrightarrow \cm{n+1} \in (\Q{1-\alpha-2\epsilon}{F} - 2\gam_2, \Q{1-\alpha+\epsilon}{F} + \gam_2].
    \end{align*}
    Exchanging the roles of $\PI{\cdot}$ and $\PIs{\cdot}$ and using the same arguments, we conclude that also
    \begin{align*}
        y_{n+1} \in \PI{\gam} \Delta \PIs{\gam}
        \Longrightarrow \cm{n+1} \in (\Q{1-\alpha-2\epsilon}{F} - 2\gam_2, \Q{1-\alpha+\epsilon}{F} + \gam_2]
    \end{align*}
    holds.

    To sum it up, we have shown that
    \begin{align*}
        \P(y_{n+1} \in \PI{\gam} \Delta \PIs{\gam})
        &\leq \P(\cm{n+1} \in (\Q{F}{1-\alpha-2\epsilon} - 2\gam_2, \Q{1-\alpha+\epsilon}{F} + \gam_2]) \\
        &+ \P(\ld{\gam_2}(\hat{F}, F) > \epsilon) 
        + \P(\ld{\gam_2}(\hat{F}, \hat{G}) > \epsilon).
    \end{align*}
    As in the proofs of \Cref{thm:ma_asymptoticResult} and \Cref{thm:sc_equivalence} we can show that
    \begin{align*}
        \limN \P(\ld{\gam_2}(\hat{F}, F) > \epsilon) + \P(\ld{\gam_2}(\hat{F}, \hat{G}) > \epsilon) = 0.
    \end{align*}
    With a similar argument as in \cref{eq:app_DeltaContinuity} we have
    \begin{align*}
        &\P(\cm{n+1} \in (\Q{1-\alpha-2\epsilon}{F} - 2\gam_2, \Q{1-\alpha+\epsilon}{F} + \gam_2])\\
        &\leq \E(\min(1, 3 \gam_2 \supnorm{f_{\Tn}})
        + \P(\cm{n+1} \in (\Q{1-\alpha-2\epsilon}{F}, \Q{1-\alpha+\epsilon}{F}]),
    \end{align*}
    where the last summand can be bounded from above by $3\epsilon$ (even for $\epsilon > \min(\frac{1-\alpha}{2}, \alpha)$. 
    Thus, we have shown that
    \begin{align*}
        \limsup_{n \to \infty}
        \P(y_{n+1} \in \PI{\gam} \Delta \PIs{\gam}) 
        \leq 3 \epsilon + \limsup_{n \to \infty}\E(\min(1, 3 \gam_2 \supnorm{f_{\Tn}})
    \end{align*}
    for all $\epsilon > 0$ and $\gam_2 > 0$. Since $\supnorm{f_{\Tn}}$ is bounded in probability and $\epsilon > 0$ and $\gam_2 > 0$ can be made arbitrarily small, the claim follows.\footnote{Note that we could also have argued that there exists null-sequences $(\epsilon_n)_{n \in \N}$ and $(\gam_n)_{n \in \N}$ such that $\limN \P(\ld{\gam_n}(\hat{F}, \hat{G}) > \epsilon) = 0$ and the same holds true if we replace $\hat{G}$ by $F$.}

\end{proof}

\begin{proof}[Proof of \Cref{lem:dis_InToOut}]
    Let $I_n(x) = \{i \in \{1, \ldots, n\}: x_i = x\}$.
    We then define the new algorithm as follows:
    \begin{align*}
        \mathcal{B}(x, \Tn) =
        \begin{cases}
            \frac{1}{|I_n(x)|}\sum_{i \in I_n(x)} \mathcal{A}(x, \Tnlt{n+1}{i}) & \text{if } I_n(x) \neq \emptyset \\
            \mathcal{A}(x, \Tn) & \text{else.}
        \end{cases}
    \end{align*}
    Since the distribution of $x_{n+1}$ is nonatomic and the $x_i$ are independent of each other, we immediately conclude
    $I_n(x_i) = \{i\}$ whenever $i \leq n$ and $I_n(x_{n+1}) = \emptyset$ almost surely for all $n \in \N$.
    Denoting the training data of size $n+1$ where the $i$-th data point $(y_i, x_i)$ is excluded with $\Tnp^{\backslash i}$, we conclude
    $$\mathcal{B}(x_{n+1}, \Tn) = \mathcal{A}(x_{n+1}, \Tn) \text{ and } \mathcal{B}(x_i, \Tn) = \mathcal{A}(x_i, \Tnlt{n+1}{i}) = \mathcal{B}(x_i, \Tnlt{n+1}{i})$$
    almost surely for all $n \geq i$.
    
    Since the data are i.i.d., exchanging $(y_i, x_i)$ with $(y_{n+1}, x_{n+1})$ shows that the distribution of $\mathcal{A}(x_{n+1}, \Tn) - \mathcal{A}(x_{n+1}, \Tnlt{n+1}{i})$ coincides with that of $\mathcal{A}(x_i, \Tnp^{\backslash i}) - \mathcal{A}(x_i, \Tnlt{n+1}{i})$ for all $n \geq i$.
\end{proof}

\begin{proof}[Proof of \Cref{thm:app_fcBecomesJackknife}]
    Define $\mathcal{B}$ as in the proof of \Cref{lem:dis_InToOut}, that is,
    \begin{align*}
        \mathcal{B}(x, \Tn) =
        \begin{cases}
            \frac{1}{|I_n(x)|}\sum_{i \in I_n(x)} \mathcal{A}(x, \Tnlt{n+1}{i}) & \text{if } I_n(x) \neq \emptyset \\
            \mathcal{A}(x, \Tn) & \text{else,}
        \end{cases}
    \end{align*}
    where $I_n(x)$ denotes the set $\{i \in \{1, \ldots, n\}: x_i = x\}$.

    By \Cref{lem:dis_InToOut} we then have $\mathcal{B}(x_i, \Tn) = \mathcal{A}(x_i, \Tnlt{n+1}{i})$ and $\mathcal{B}(x_{n+1}, \Tnp^y) = \mathcal{A}(x_{n+1}, \Tn)$ almost surely. 
    Thus, $\PIs{\gam}$ applied to the algorithm $\mathcal{B}$ coincides with the set $[\mathcal{A}(x_{n+1}, \Tn) \pm (\Q{\hat{G}}{1-\alpha} + \gam)]$,
    where $\hat{G}$ denotes the empirical distribution function of the absolute leave-one-out residuals of $\mathcal{A}$. Thus, $\PIs{\gam}$ applied to the algorithm $\mathcal{B}$ coincides with the symmetric Jackknife prediction interval applied to the original algorithm $\mathcal{A}$. 

    Furthermore, we have
    \begin{align*}
        &|\cm{n+1} - \mathcal{C}_n(t_{n+1}, \Tnlt{n+1}{n})|
        = \left| |y_{n+1} - \mathcal{B}(x_{n+1}, \Tnp)| 
        - |y_{n+1} - \mathcal{B}(x_{n+1}, \Tnl{n})| \right| \\
        &\leq |\mathcal{B}(x_{n+1}, \Tnp) - \mathcal{B}(x_{n+1}, \Tnl{n})|
        = |\mathcal{A}(x_{n+1}, \Tn) - \mathcal{A}(x_{n+1}, \Tnlt{n+1}{n}|.
    \end{align*}
    Since $\mathcal{A}$ is asymptotically out-of-sample stable, the instability coefficient $|\cm{n+1} - \mathcal{C}_n(t_{n+1}, \Tnlt{n+1}{n})|$ converges to $0$ in probability.
    Thus, we can apply \Cref{thm:sc_equivalence}, which proves the equivalence between $\PI{\gam}$ and $\PIs{\gam}$.
\end{proof}

\begin{proof}[Proof of \Cref{lem:dis_OutToIn}]
    Let $I_n(x) = \{i \in \{1, \ldots, n\}: x_i = x\}$.
    We then define the new algorithm as follows:
    \begin{align*}
        \mathcal{B}(x, \Tn) =
        \begin{cases}
            \frac{1}{|I_n(x)|}\sum_{i \in I_n(x)} \mathcal{A}(x, \Tn) & \text{if } I_n(x) \neq \emptyset \\
            \frac{1}{n} \sum_{i=1}^n \mathcal{A}(x_i, \Tn) & \text{else.}
        \end{cases}
    \end{align*}
    Since the distribution of $x_{n+1}$ is nonatomic and the $x_i$ are independent of each other, we immediately conclude
    $I_n(x_i) = \{i\}$ whenever $i \leq n$ and $I_n(x_{n+1}) = \emptyset$ almost surely for all $n \in \N$. This shows that $\mathcal{B}(x_{i}, \Tn) = \mathcal{A}(x_i, \Tn)$ almost surely for all $1 \leq i \leq n$.

    For the second part of the proof, we note that
    \begin{align*}
        &\E \left| \mathcal{B}(x_{n+2}, \Tnp) - \mathcal{B}(x_{n+2}, \Tn) \right|
        = \E \left| \dfrac{1}{n+1} \sum_{i=1}^{n+1} \mathcal{A}(x_{i}, \Tnp) - \dfrac{1}{n} \sum_{i=1}^n \mathcal{A}(x_{i}, \Tn) \right| \\
        &\leq \dfrac{1}{n+1} \sum_{i=1}^n \E \left| \mathcal{A}(x_{i}, \Tnp) - \mathcal{A}(x_{i}, \Tn) \right| + \dfrac{1}{n(n+1)} \sum_{i=1}^n \E \left| \mathcal{A}(x_{n+1}, \Tnp) - \mathcal{A}(x_{i}, \Tn) \right|.
    \end{align*}
    For $1 \leq i \leq n$, we conclude 
    \begin{align*}
        \E \left| \mathcal{A}(x_{i}, \Tn) - \mathcal{A}(x_{n+1}, \Tnp) \right|
        &\leq \E \left| \mathcal{A}(x_{i}, \Tn) - \mathcal{A}(x_{i}, \Tnp) \right|\\
        &+ \E \left| \mathcal{A}(x_{i}, \Tnp) - \mathcal{A}(x_{n+1}, \Tnp) \right|.
    \end{align*}
    Since the algorithm is symmetric, we have 
    $$\E \left| \mathcal{A}(x_{i}, \Tnp) - \mathcal{A}(x_{n+1}, \Tnp)\right| = \E \left| \mathcal{A}(x_{1}, \Tnp) - \mathcal{A}(x_{2}, \Tnp)\right|.$$
    Putting the pieces together, we end up with
    \begin{align*}
        \E \left| \mathcal{B}(x_{n+2}, \Tnp) - \mathcal{B}(x_{n+2}, \Tn) \right|
        &\leq \dfrac{1}{n} \sum_{i=1}^n \E \left| \mathcal{A}(x_{i}, \Tnp) - \mathcal{A}(x_{i}, \Tn) \right| \\
        &+ \dfrac{1}{n+1}\E \left| \mathcal{A}(x_{1}, \Tnp) - \mathcal{A}(x_{2}, \Tnp)\right|.
    \end{align*}
    
\end{proof}

\begin{proof}[Proof of \Cref{prop:app_algorithm}]
    We start proving the number of refits:
    Consider \Cref{alg:bisection} applied to some $L, U \in \R$ and $\epsilon > 0$ and note that in each step the algorithm has to be refitted only once.
    Since the interval length is halved after each step, it takes no more than $\max(0, \lceil \log_2(L-U) - \log_2(\epsilon)\rceil)$ steps until the new length is less than (or equal to) $\epsilon$.

    Next, we consider \Cref{alg:minimizer} applied with $L, U \in \R$ and $\epsilon > 0$.
    For the calculation of $c_L, c_U, c_s, c_t$, the algorithm has to be refitted four times. 
    Since after each step the new interval length has been shrunken by the factor $g = \frac{\sqrt{5}-1}{2}$, 
    it takes no more than $\max(0, \lceil \frac{\log_2(L-U) - \log_2(\epsilon)}{\log_2(g)} \rceil)$ steps until the new length is less than (or equal to) $\epsilon$.
    Thus, the algorithm needs to be refitted no more than $4 + \max(0, \lceil \frac{\log_2(L-U) - \log_2(\epsilon)}{\log_2(g)} \rceil)$ times.

    Finally, we consider \Cref{alg:main} applied with boundary parameter $K \in \Z$ and $\epsilon \in (0, 2^K]$.
    In particular, we have $K \geq \log_2(\epsilon)$.
    For the calculation of $b, c_1, c_2$, the algorithm has to be refitted three times.
    If either $c_1 \leq b < c_2$ or $c_2 \leq b < c_1$, then the algorithm finishes after applying \Cref{alg:bisection}. Since the length of the interval is $2^{K+1}$, \Cref{alg:bisection} takes no more than $\lceil K+1 - \log_2(\epsilon)\rceil$ steps, which proves the claim.

    If $\max(c_1, c_2) \leq b$, the algorithm immediately stops and the claim of \Cref{prop:app_algorithm} is also fulfilled.

    It remains to deal with the case where both $c_1$ and $c_2$ are larger than $b$. 
    In this situation, we apply \Cref{alg:minimizer} to $L = -2^K$ and $U = 2^K$, which takes no more than
    $4 + \lceil \frac{K+1 - \log_2(\epsilon)}{-\log_2(g)} \rceil$ refits of the algorithm.
    Now the algorithm terminates unless $|m| < 2^K$. In this situation, the model is refitted once (for $m$).
    Additionally, \Cref{alg:bisection} may be applied twice which gives
    \begin{align*}
        \max(0, \lceil \log_2(2^K - m) - \log_2(\epsilon) \rceil) + 
        \max(0, \lceil \log_2(m + 2^K) - \log_2(\epsilon) \rceil)
    \end{align*}
    refits.
    If either $2^K - m < \epsilon$ or $m + 2^K < \epsilon$ this can be bounded from above by 
    $\lceil \log_2(2^{K+1}) - \log_2(\epsilon) \rceil) \leq K+2 - \log_2(\epsilon)$.
    In the case where $2^K -m \geq \epsilon$ and $m + 2^K \geq \epsilon$ we have
    \begin{align*}
        &\max(0, \lceil \log_2(2^K - m) - \log_2(\epsilon) \rceil) + 
        \max(0, \lceil \log_2(m + 2^K) - \log_2(\epsilon) \rceil)\\
        &\leq 2 + \log_2(\frac{2^K - m}{\epsilon}) + \log_2(\frac{m + 2^K}{\epsilon})
        = 2 + \log_2 \left( \dfrac{2^{2K} - m^2}{\epsilon^2} \right) 
        \leq 2(1 + K - \log_2(\epsilon)). 
    \end{align*} 
    
    Counting everything together, the algorithm has to be refitted less than or equal to
    \begin{align*}
        &3 + \left(4 + \left\lceil \frac{K+1 - \log_2(\epsilon)}{-\log_2(g)} \right\rceil \right)
        + 1 + 2(1 + K - \log_2(\epsilon)) \\
        &\leq 10 + \left(K+1-\log_2(\epsilon)\right)\left(2 + \dfrac{1}{\log_2(\phi)}\right),
    \end{align*}
    where we used the fact that $\phi = \frac{1}{g}$.
    
    Next, we proof $\PIs{\gam} \subseteq PI$. First, we note that, due to the unimodality, $\PIs{\gam}$ is a (possibly unbounded) interval and distinguish several cases:
    If $\PIs{\gam} = \emptyset$, then $\PIs{\gam} \subseteq PI$ is trivially fulfilled.
    If $\PIs{\gam} \supseteq [-2^K, 2^K]$, then $\mathcal{C}_n((y, x_{n+1}), \Tn) \leq \Q{1-\alpha}{\hat{G}}$ for all $y \in [-2^K, 2^K]$. In particular, in \Cref{alg:main} we have $\max(c_1, c_2) \leq b$ and hence the algorithm gives $PI = \R$, trivially fulfilling $\PIs{\gam} \subseteq PI$.
    If $\PIs{\gam} \subseteq (-2^K, \infty)$ and $2^K \in \PIs{\gam}$, we have $c_2 \leq b < c_1$. In this case, \Cref{alg:bisection} computes a $u \in [-2^K, 2^K]$ such that $u \notin \PIs{\gam}$. In particular, we $u < y$ for all $y \in \PIs{\gam}$ and therefore $\PIs{\gam} \subseteq [u, \infty) = PI$.
    The case $\PIs{\gam} \subseteq (-\infty, 2^K)$ and $-2^K \in \PIs{\gam}$ can be shown analogously.
    If $\emptyset \neq \PIs{\gam} \subseteq (-\infty, -2^K)$, the conformity score is nondecreasing on $[-2^K, 2^K]$ and therefore $m = -2^K$. In this case, we have $\PIs{\gam} \subseteq (-\infty, -2^K + \epsilon) = PI$.
    The case $\emptyset \neq \PIs{\gam} \subseteq (2^K, \infty)$ can be proven analogously.
    
    It remains to deal with the case $\emptyset \neq \PIs{\gam} \subseteq [-2^K, 2^K]$.
    If \Cref{alg:minimizer} gives the output $m = -2^K$, then, by the unimodality we conclude 
    $$\mathcal{C}_n((y, x_{n+1}), \Tn) > \mathcal{C}_n((-2^K, x_{n+1}), \Tn) > b$$
    for all $y \geq -2^K+\epsilon$. To put it in other words, we have $\PIs{\gam} \subseteq (-2^K, -2^K + \epsilon) \subseteq PI = (-\infty, -2^K + \epsilon)$.
    The case $m = 2^K$ can be shown analogously.

    Now, if $|m| < 2^K$, we conclude that the minimum of the conformity score is indeed located in $(-2^K, 2^K)$. 
    If $m \notin \PIs{\gam}$, we also have $M \notin \PIs{\gam}$ and, by \Cref{alg:minimizer} the minimum of the conformity score is inside the interval $(\min(m,M), \max(m,M))$. Consequently, we have $\PIs{\gam} \subseteq (\min(m,M), \max(m,M))$. Furthermore, we have $\lambda(PI \backslash \PIs{\gam}) \leq \lambda(PI) \leq \epsilon$.

    In the last remaining case we have $m \in \PIs{\gam} \subseteq [-2^K, 2^K]$ and $|m| < 2^K$.
    In this situation, \Cref{alg:bisection} calculates points $l_1 \in (-2^K, m)$ and $u_1 \in (m, 2^K)$, such that $l_1, u_1 \notin \PIs{\gam}$.
    Since $m \in \PIs{\gam}$, this shows that $\PIs{\gam} \subseteq PI$.

    It remains to show that in the case $\emptyset \neq \PIs{\gam} \subseteq [-2^K+\epsilon, 2^K-\epsilon]$ we have $[L+\epsilon, U-\epsilon] \subseteq \PIs{\gam} \subseteq \PIs{\gam}$.
    In this case, the minimum of the conformity score is located inside of $[-2^K+\epsilon, 2^K-\epsilon]$ and therefore $|m| < 2^K$. 
    In the case where $m \notin \PIs{\gam}$, we have $\lambda(PI) \leq \epsilon$ as shown before and hence $[L+\epsilon, U - \epsilon] = \emptyset \subseteq \PIs{\gam}$ trivially.
    In the case where $m \in \PIs{\gam}$, the bisection algorithm provides $l_1$ and $l_0$, such that $l_1 \notin \PIs{\gam}$ and $l_0 \in \PIs{\gam}$ with $l_0 \leq l_1 + \epsilon$. Similarly, we have $u_0 \in \PIs{\gam}$ and $u_1 \notin \PIs{\gam}$ with $u_0 \geq u_1 - \epsilon$.
    Thus, the set $[l_0, u_0]$ is a subset of $\PIs{\gam}$ (even in the case $l_0 > u_0$).
    As a consequence we have $[L+\epsilon, U-\epsilon] \subseteq \PIs{\gam}$ (even in the case $U - L < 2\epsilon$).
\end{proof}